\pdfoutput=1
\documentclass[final,US,11pt,reqno,bibtex,nopagebackref,notcite,notref]{amsart}
\usepackage{liao}

\DeclareMathOperator{\pr}{pr}
\DeclareMathOperator{\im}{im}

\DeclareMathOperator{\End}{End}
\DeclareMathOperator{\Aut}{Aut}

\DeclareMathOperator{\id}{id}
\DeclareMathOperator{\rk}{rk}

\DeclareMathOperator{\Bott}{Bott}

\DeclareMathOperator{\ver}{vertical}

\DeclareMathOperator{\pbw}{pbw}
\DeclareMathOperator{\formal}{formal}

\newcommand{\NN}{\mathbb{N}} 
 
\newcommand{\ZZ}{\mathbb{Z}} 
 
\newcommand{\RR}{\mathbb{R}} 
\newcommand{\CC}{\mathbb{C}} 
\newcommand{\KK}{\Bbbk}
 
\newcommand{\into}{\hookrightarrow}
\newcommand{\onto}{\twoheadrightarrow}
\newcommand{\xto}[1]{\xrightarrow{#1}}
\newcommand{\dual}{^{\vee}}
\newcommand{\inv}{^{-1}}

\newcommand{\frakg}{\mathfrak g}
\newcommand{\frakh}{\mathfrak h}

\newcommand{\cM}{\mathcal{M}}

\newcommand{\cR}{\mathcal{R}}
\newcommand{\cO}{\mathcal{O}}

\newcommand{\cF}{\mathcal{F}}
\newcommand{\cU}{\mathcal{U}}
\newcommand{\cI}{\mathcal{I}}

\newcommand{\sD}{\mathscr{D}}
\newcommand{\sE}{\mathscr{E}}

\newcommand{\pair}[2]{\langle #1, #2 \rangle}

\newcommand{\argument}{-}

\newcommand{\koszul}{\mathsf{k}}
\newcommand{\euler}{\epsilon}
\newcommand{\htp}{\mathsf{h}}
\newcommand{\htpd}{\mathsf{h}_\natural}
\newcommand{\vdop}{\sD_{\ver}}

\usepackage{scalerel}

\newcommand{\pa}{\mathfrak{p}} 
\newcommand{\ia}{\mathfrak{i}} 
\newcommand{\pb}{\mathfrak{q}} 
\newcommand{\ib}{\mathfrak{j}} 

\newcommand{\hotimes}{\mathbin{\widehat{\otimes}}}

\newcommand{\tail}[1]{\Gamma(\Lambda^\bullet L\dual\otimes\widehat{S}^{\geq #1}B\dual)}
\newcommand{\homogeneous}[1]{\Gamma(\Lambda^\bullet L\dual\otimes S^{#1}B\dual)}
\newcommand{\head}[1]{\Gamma(\Lambda^\bullet L\dual\otimes S^{< #1}B\dual)}
\newcommand{\whole}{\Gamma(\Lambda^\bullet L\dual\otimes\widehat{S}B\dual)}

\title{Vertical isomorphisms of Fedosov dg manifolds associated with a Lie pair}

\thanks{Research partially supported by MoST/NSTC Grants 110-2115-M-007-001-MY2 and 112-2115-M-007-016-MY3.}

\author{Hua-Shin Chang}
\address{Department of Mathematics, National Tsing Hua University}
\email{eason901001@gmail.com } 

\author{Hsuan-Yi Liao}
\address{Department of Mathematics, National Tsing Hua University}
\email{hyliao@math.nthu.edu.tw}

\begin{document}

\begin{abstract}
We investigate vertical isomorphisms of Fedosov dg manifolds associated with a Lie pair $(L,A)$, i.e.\ a pair of a Lie algebroid $L$ and a Lie subalgebroid $A$ of $L$. The construction of Fedosov dg manifolds involves a choice of a splitting and a connection. We prove that, given any two choices of a splitting and a connection, there exists a unique vertical isomorphism, determined by an iteration formula, between the two associated Fedosov dg manifolds. As an application, we provide an explicit formula for the map $\pbw_2\inv \circ \pbw_1$ associated with two Poincar\'{e}--Birkhoff--Witt isomorphisms that arise from two choices of a splitting and a connection.
\end{abstract}

\maketitle

\tableofcontents

\section*{Introduction}

Fedosov resolutions, which can be interpreted as functions on Fedosov dg manifolds, played a crucial role in globalizing Kontsevich's formality theorem to smooth manifolds \cite{MR2102846,MR2062626,MR2199629}. The construction of Fedosov dg manifolds was originally introduced in \cite{MR1327535} for smooth manifolds, as a classical analogue of Fedosov's quantization \cite{MR1293654}, and has been extended to Lie pairs and to dg manifolds as an important step of establishing formality theorems and Duflo--Kontsevich-type theorems for Lie pairs and for dg manifolds --- see \cite{MR3964152,MR3754617,MR4584414}.  Furthermore, the consideration of Fedosov dg manifolds has also  led to new insights into Atiyah classes and Todd classes \cite{MR3724780,MR3964152}. 

By a {\em Lie pair} $(L,A)$, we mean a pair consisting of a Lie algebroid $L$ over a smooth manifold $M$ and a Lie subalgebroid $A$ of $L$. The structure of Lie pairs arises naturally in various geometrical contexts. For instance, a regular foliation $\cF$ of a smooth manifold $M$ determines a Lie pair $(T_M,T_\cF)$, where $T_M$ is the tangent bundle of $M$ and $T_\cF$ is the integrable distribution on $M$ tangent to $\cF$. Another example arises from complex geometry: if $X$ is a complex manifold, then $(T_X \otimes \CC, T_X^{0,1})$ is a Lie pair. Notably, Lie pairs also offer a framework for studying Lie algebra actions on a smooth manifold \cite{MR3650387,MR3964152}. 

Let $(L,A)$ be a Lie pair over a smooth manifold $M$, and $B$ the quotient vector bundle $B = L/A$. 
In \cite{MR4150934}, it was proved that for each choice of (i) a splitting of the short exact sequence of vector bundles 
\begin{equation}\label{intro:ShortExactSeq}
0 \to A \to L \to B \to 0
\end{equation}
and (ii) a torsion-free $L$-connection $\nabla$ on $B$, one can construct a homological vector field of the form
$$
Q = -\koszul + d_L^\nabla +X,
$$
called a \emph{Fedosov vector field}, on the (formal) graded manifold $L[1] \oplus B$ by Fedosov's iteration techniques \cite{MR1293654}.  
Here, $\koszul$ is the Koszul vector field defined in $\eqref{eq:KoszulVF}$
and $d_L^\nabla$ is the covariant derivative
associated with the $L$-connection $\nabla$. 
This construction results in a cochain complex $(C^\infty(L[1] \oplus B),Q)$ that is quasi-isomorphic to the Chevalley--Eilenberg complex $(\Gamma(\Lambda^\bullet A\dual),d_A)$ of the Lie algebroid $A$.
See Section~\ref{sed:FedosovDGLiePair} for details. 
The pair  $(L[1] \oplus B, Q)$ is referred to as a \emph{Fedosov dg manifold} associated with the Lie pair $(L,A)$. 

The algebra $\cR = \Gamma(\Lambda^\bullet L\dual \otimes \widehat{S}B\dual)$
of functions on $L[1] \oplus B$ is the $\cI$-adic completion of the graded algebra $\Gamma(\Lambda^\bullet L\dual \otimes {S}B\dual)$. Here, $\cI$ denotes the ideal generated by $\Gamma(\Lambda^\bullet L\dual \otimes B\dual)$. 
An endomorphism of the algebra $\cR$ is called \emph{vertical} if it is $\Gamma(\Lambda^\bullet L\dual)$-linear and continuous with respect to the $\cI$-adic topology. Now assume we are given two different choices of a splitting of \eqref{intro:ShortExactSeq} and a torsion-free $L$-connection on $B$, leading to two Fedosov vector fields $Q_1$ and $Q_2$ on $L[1] \oplus B$. The main purpose of this paper is to investigate \emph{vertical isomorphisms of Fedosov dg manifolds} from $(L[1] \oplus B,Q_1)$ to $(L[1] \oplus B,Q_2)$, i.e.\ {\em vertical isomorphisms} of dg algebras from $(\cR,Q_2)$ to $(\cR,Q_1)$. 

The main result of this paper is the following theorem:

\begin{trivlist}
\item {\bf Theorem} (Theorem~\ref{thm:PhiIteration} \& Corollary~\ref{cor:Y=LogPhi}).
{\it Given any two choices of a splitting of the short exact sequence \eqref{intro:ShortExactSeq} and a torsion-free $L$-connection on $B$
leading to two Fedosov vector fields $Q_1$ and $Q_2$ on $L[1]\oplus B$, there exists a unique vertical isomorphism of Fedosov dg manifolds
$\phi:(\cR,Q_2) \to (\cR,Q_1)$.
This vertical isomorphism $\phi$ is characterized by the equation
\begin{equation}\label{intro:PhiIteration}
\phi = 1 + \htpd\eth(\phi) = 1+ \htpd (\Delta Q \circ \phi)+ \htpd ([\koszul + Q_2 ,\phi]), 
\end{equation}
where $\Delta Q=Q_1-Q_2$ and $\eth(\phi)=\Delta Q\circ\phi+[\koszul+Q_2,\phi]$. 
Furthermore, we have $\phi = e^Y = \sum_{n=0}^\infty\frac{1}{n!} Y^n$, where $Y  \in \Gamma(\widehat{S}^{\geq 2}B\dual \otimes B)$ is given by the formula
\begin{equation}\label{intro:Y=LogPhi}
Y = \sum_{j=1}^\infty \sum_{n_1, \cdots, n_j \geq 0} \frac{(-1)^{j+1}}{j} \cdot (\htpd \eth)^{n_1}\big(\htpd (\Delta Q)\big) \circ (\htpd \eth)^{n_2}\big(\htpd (\Delta Q)\big) \circ \cdots \circ (\htpd \eth)^{n_j}\big(\htpd (\Delta Q)\big)  .
\end{equation}
} 
\end{trivlist}

We say that a vertical operator $\phi:\cR\to\cR$ {shifts the $\cI$-adic filtration} by $N \in \ZZ$ (or simply call it an \emph{$N$-shifting} operator) if  
$$
\phi\big(\Gamma(\Lambda^\bullet L\dual \otimes S^k B\dual)\big) \subset \Gamma(\Lambda^\bullet L\dual \otimes \widehat{S}^{\geq k+N} B\dual), \quad \forall \, k \geq 0.
$$
A vertical operator is called an \emph{$\cI$-adic filtration-shifting} vertical operator if it shifts the $\cI$-adic filtration by a certain integer $N$, and we denote the space of all $\cI$-adic filtration-shifting vertical operators
by the symbol $\sE$. 
Our main theorem above is a consequence of the existence of the Koszul-type
homotopy operator $\htpd$ acting on $\sE$. 
The key observations are: (i) there exists a Koszul-type homotopy equation for \emph{vertical differential operators} \cite{MR4325718,MR3964152}, and (ii) if $\phi$ is an $N$-shifting vertical operator, then it can be uniquely decomposed as an (infinite) sum of vertical differential operators $D_q \in \Gamma(\Lambda^\bullet L\dual \otimes \widehat{S}^{\geq q+N} B\dual \otimes S^q B)$:
$$
\phi = \sum_{q=0}^\infty D_q,
$$ 
where the equality means that the series $\sum_{q=0}^\infty D_q$
converges uniformly to $\phi$ with respect to the $\cI$-adic topology
--- see Proposition~\ref{prop:BddOp&DiffOp}. 
Therefore, by applying the homotopy equation to each summand, one can extend the Koszul-type homotopy equation to $\cI$-adic filtration-shifting vertical operators, i.e.\ there exist a Koszul-type homotopy operator 
$$
\htpd:\sE \to \sE, \qquad \phi \mapsto \sum_{q=0}^\infty \htpd(D_q),
$$ 
and a Koszul-type homotopy equation for $\cI$-adic filtration-shifting vertical operators:
$$
\delta \htpd (\phi)+\htpd\delta(\phi) = \phi - \sigma_0(\phi), \qquad \forall \, \phi \in \sE.
$$
Here, $\delta=[\koszul,\argument]$ is the differential
given by the graded commutator with the Koszul vector field $\koszul$,
and $\sigma_0$ is a projection operator.
See Section~\ref{sec:BddBelowVerOp} for details.

\subsubsection*{Application to Kapranov dg manifolds}

Given two different choices of a splitting of the short exact sequence 
\eqref{intro:ShortExactSeq} and a torsion-free $L$-connection on $B$,
one obtains two induced PBW (i.e.\ Poincar\'{e}--Birkhoff--Witt) isomorphisms
$\pbw_1$ and $\pbw_2: \Gamma(SB) \to \frac{\cU(L)}{\cU(L)\Gamma(A)}$.
Due to Sti\'{e}non--Xu's PBW construction of Fedosov vector fields \cite[Theorem~4.7]{MR4150934}, it can be shown that the map 
$$
\phi= \id_{\Lambda^\bullet L\dual}\otimes (\pbw_2\inv\circ \pbw_1)\dual: \Gamma(\Lambda^\bullet L\dual \otimes \widehat{S} B\dual) \to \Gamma(\Lambda^\bullet L\dual \otimes \widehat{S} B\dual)
$$ 
is a vertical isomorphism of Fedosov dg manifolds from $(L[1]\oplus B,Q_1)$ to $(L[1]\oplus B,Q_2)$.   See \cite{MR4325718,MR4150934}. 
Consequently, there exists a unique $Y\in \Gamma(\widehat{S}^{\geq 2}B\dual \otimes B)$ determined by Equation~\eqref{intro:Y=LogPhi} such that $\phi = e^Y$. In other words, 
\begin{equation}\label{intro:pbw2pbw1}
\pair{e^Y (f)}{\theta}= \pair{f}{\pbw_2\inv \circ \pbw_1(\theta)}, \qquad \forall \, f \in \Gamma(\widehat{S}B\dual), \, \theta \in \Gamma(SB).
\end{equation} 

In the case $(L,A) = (T_M,M\times 0)$, the map $\pbw: \Gamma(ST_M) \to \cU(T_M)$ is the fiberwise $\infty$-jet of the geodesic exponential map $\exp:T_M \to M \times M$ along the zero section of $T_M$. See \cite{MR4271478,MR3910470}. Thus, Equation~\eqref{intro:pbw2pbw1} computes the $\infty$-jet of $\exp_{2}\inv \circ \exp_{1}$, which is the coordinate transformation between two geodesic coordinate systems on $M$. 

In \cite{MR4271478}, Laurent-Gengoux, Sti\'{e}non and Xu investigated Kapranov dg manifolds associated with a Lie pair $(L,A)$. The construction of Kapranov dg manifolds depends on a choice of a splitting of the short exact sequence \eqref{intro:ShortExactSeq} and a torsion-free $L$-connection on $B$. According to \cite[Proposition~5.24]{MR4271478}, the isomorphism $\pbw_2\inv \circ \pbw_1$ induces an isomorphism between Kapranov dg manifolds arising from two choices. Therefore, Equation~\eqref{intro:pbw2pbw1} provides an explicit formula of an isomorphism between Kapranov dg manifolds that result from two choices of a splitting and a connection.

\subsection*{Notations and conventions}

We use the symbol $\NN$ to denote the set of positive integers and the symbol $\NN_0$ for the set of nonnegative integers.

We fix a base field $\KK = \RR$ or $\CC$ in this paper. The notation $C^\infty(M)=C^\infty(M,\KK)$ refers to the algebra of smooth functions on a manifold $M$ valued in $\KK$, and $T_M$ refers to $T_M \otimes_\RR \KK$ unless stated otherwise. 

In this paper, grading means $\ZZ$-grading.
We write `dg' for `differential graded.'

Let $R$ be a graded ring. We say $R$ is commutative if $xy = (-1)^{|x||y|} yx$ for all homogeneous elements $x,y \in R$. 

Given an element $x$ in a graded $R$-module $V=\bigoplus_k V^k$,
the notation $|x|=d$ means that $x$ is a homogeneous element
whose degree is $d$, i.e. $x\in V^d$.
We write $V[i]$ to denote the graded $R$-module obtained from $V$
by shifting the grading: $(V[i])^k = V^{i+k}$.

For an $R$-module $V$, the symbol $S^k V$ means the zero $R$-module if $k < 0$. 

By the symbol $\Lambda^\bullet V$, we mean $\bigoplus_p (\Lambda^p V)[-p]$ which is isomorphic to $S(V[-1])$ as graded modules.

For every vector bundle $E \to M$, we define a duality pairing $\pair{\argument}{\argument}:\Gamma(SE) \times \Gamma(\widehat{S}E\dual) \to C^\infty(M)$ by 
$$
\pair{e_1 \odot \cdots \odot e_i}{\epsilon_1 \odot \cdots \odot \epsilon_j} = \begin{cases}
\sum_{\sigma \in S_i} \prod_{k=1}^j \pair{e_k}{\epsilon_{\sigma(k)}} & \text{ if } i=j, \\
0 & \text{ otherwise},
\end{cases}
$$
where $e_1, \cdots, e_i \in \Gamma(E)$ and $\epsilon_1, \cdots, \epsilon_j \in \Gamma(E\dual)$.

A $(i,j)$-shuffle is a permutation of the set $\{1,2,\cdots, i+j\}$ such that 
$$
\sigma(1)< \cdots < \sigma(i) \qquad \text{and} \qquad \sigma(i+1) < \cdots < \sigma(i+j).
$$
The symbol $S_i^j$ denotes the set of $(i,j)$-shuffles. 

Let $\eta^1, \cdots, \eta^r$ be a local coordinate system on a graded manifold. Given a multi-index $J = (j_1, \cdots, j_r) \in \NN_0^r$, we use the following multi-index notations:
\begin{align*}
|J| & = j_1+j_2+\cdots + j_r, \\
\eta^J & = (\eta^1)^{j_1} \cdot (\eta^2)^{j_2} \cdot   \cdots \cdot (\eta^r)^{j_r}, \\
\partial_J  & = \partial_{\eta,J} = \frac{\partial^{j_1}}{\partial (\eta^1)^{j_1}} \circ \frac{\partial^{j_2}}{\partial (\eta^2)^{j_2}} \circ \cdots \circ \frac{\partial^{j_r}}{\partial (\eta^r)^{j_r}}.
\end{align*}

\subsection*{Acknowledgments}

We would like to thank Camille Laurent-Gengoux, Seokbong Seol, Mathieu Sti\'{e}non, Si-Ye Wu and Ping Xu for fruitful discussions and useful comments. 
Chang is grateful to National Center for Theoretical Science for organizing undergraduate research programs.

\section{Preliminaries}

In this section, we briefly review the construction of Fedosov dg manifolds. We refer the reader to \cite{MR4150934} for details.

\subsection{Preliminaries on Lie pairs}

We say that $(L,A)$ is a \textbf{Lie pair} if $A$ is a Lie subalgebroid of a Lie algebroid $L$ over a common base manifold $M$, i.e. the inclusion map $A \into L$ is a morphism of Lie algebroids.

\begin{example}
Following are a few examples of Lie pairs.
\begin{enumerate}
\item 
If $\frakh$ is a Lie subalgebra of a Lie algebra $\frakg$, then $(\frakg,\frakh)$ is a Lie pair over the one-point manifold $\{\ast\}$.
\item
If $\cF$ is a regular foliation on a manifold $M$, then $(T_M,T_\cF)$ is a Lie pair over $M$. 
\item
If $X$ is a complex manifold, then $(T_X \otimes \CC, T_X^{0,1})$ is a Lie pair over $X$.

\item
Let $\frakg$ be a finite-dimensional Lie algebra, and $M$ a manifold. 
Given a $\frakg$-action on $M$, the action Lie algebroid $\frakg \ltimes M$ and the tangent bundle $T_M$ naturally form a matched pair of Lie algebroids \cite{MR1460632}. Thus, one has a Lie pair $(L,A)$, where $L = (\frakg \ltimes M)\bowtie T_M$ and $A = \frakg \ltimes M$. 
See \cite{MR2534186} for its relation with BRST complexes.
\end{enumerate}
\end{example}

Let $L$ be a Lie algebroid over a manifold $M$ with anchor $\rho:L \to T_M $.  Let $E $ be a vector bundle over $M$. An \textbf{$L$-connection} $\nabla$ on $E$ is a $\KK$-bilinear map 
$
\nabla: \Gamma(L) \times \Gamma(E) \to \Gamma(E), \, (l,e) \mapsto \nabla_l e
$
which satisfies the properties
\begin{gather*}
\nabla_{f\cdot l} e = f \cdot \nabla_l e, \\
\nabla_l(f \cdot e) = \rho(l)(f) \cdot e + f \nabla_l e,
\end{gather*}
for any $l \in \Gamma(L)$, $e \in \Gamma(E)$, and $f \in C^\infty (M)$.

The \textbf{Chevalley--Eilenberg differential} of a Lie algebroid $L$ is the linear map 
$$
d_L: \Gamma(\Lambda^p L\dual) \to \Gamma(\Lambda^{p+1}L\dual)
$$ 
defined by 
$$
(d_L\omega)(l_0, \cdots, l_p) = \sum_{i=0}^p (-1)^i \rho(l_i)\big(\omega(l_0, \cdots, \widehat{l_i}, \cdots, l_p)\big) + \sum_{i<j} (-1)^{i+j}\omega([l_i,l_j],l_0,\cdots, \widehat{l_i},\cdots,\widehat{l_j}, \cdots, l_p)
$$
which makes the exterior algebra $\bigoplus_{p=0}^\infty \Gamma(\Lambda^p L\dual)$ into a commutative dg algebra. Given an $L$-connection $\nabla$ on a vector bundle $E \to M$, the \textbf{covariant derivative} is the operator 
$$
d_L^\nabla:\Gamma(\Lambda^p L\dual \otimes E) \to \Gamma(\Lambda^{p+1} L\dual \otimes E)
$$
which maps a section $\omega \otimes e \in \Gamma(\Lambda^p L\dual \otimes E)$ to 
$$
d_L^\nabla(\omega\otimes e) = d_L(\omega) \otimes e + \sum_{i=1}^{\rk L} (\nu_i \wedge \omega)\otimes \nabla_{v_i} e,
$$
where $v_1,\cdots, v_{\rk L}$ is any local frame of the vector bundle $L$, and $\nu_1,\cdots, \nu_{\rk L}$ is its dual frame. The connection $\nabla$ is said to be {\bf flat} if the covariant derivative $d_L^\nabla$ is a coboundary map: $d_L^\nabla \circ d_L^\nabla =0$.

\begin{example}\label{ex:Bott}
Let $(L,A)$ be a Lie pair. 
The {\bf Bott connection} of $A$ on the quotient bundle $B =L/A$ is the flat connection 
$$
\nabla^{\Bott}: \Gamma(A) \times \Gamma(B) \to \Gamma(B)
$$ 
defined by 
$$
\nabla^{\Bott}_a(\pb(l)) = \pb([a,l]), \qquad \forall \,a \in \Gamma(A), \, l \in \Gamma(L),
$$
where $\pb:L \onto B=L/A$ is the canonical projection. 
\end{example}

Given an $L$-connection $\nabla$ on $B=L/A$, its \textbf{torsion} is the bundle map $T^\nabla: \Lambda^2 L \to B$ defined by 
$$
T^\nabla(l_1,l_2) = \nabla_{l_1} (\pb(l_2)) - \nabla_{l_2} (\pb(l_1)) - \pb([l_1,l_2]), \qquad \forall l_1,l_2 \in \Gamma(L).
$$
An $L$-connection on $B$ is called \textbf{torsion-free} if its torsion vanishes.  If an $L$-connection $\nabla$ on $B$ is torsion-free, then it extends the Bott connection, i.e.\ 
$$
\nabla_{\ia(a)} \big(\pb(l)\big) = \nabla^{\Bott}_a \big(\pb(l)\big), \qquad \forall\, a \in \Gamma(A), \, l \in \Gamma(L),
$$
where $\ia:A \to L$ is the inclusion map. See \cite[Lemma~5.2]{MR4271478}.

Let $L$ be a Lie algebroid over a manifold $M$, and $R$ denote the algebra of smooth functions on $M$. We denote by $\cU(L)$ the universal enveloping algebra of the Lie algebroid $L$. Essentially, $\cU(L)$ is the quotient of the reduced tensor algebra $\bigoplus_{n=1}^\infty \big(\big(R \oplus \Gamma(L)\big)^{\otimes_\KK n} \big)$ by the two-sided ideal generated by all the elements of the following types:
\begin{align*}
& l_1 \otimes l_2 - l_2 \otimes l_1 - [l_1,l_2],  & f \otimes l - f l, \\
& l \otimes f - f \otimes l - \rho_l(f), & f \otimes  g -fg,
\end{align*}
for $l,l_1,l_2 \in \Gamma(L)$ and $f,g \in R$. See \cite{MR0154906}.

The universal enveloping algebra $\cU(L)$ is a coalgebra over $R$. Its comultiplication $\Delta: \cU(L) \to \cU(L) \otimes_R \cU(L)$ is characterized by the equations:
\begin{gather*}
\Delta(1) = 1 \otimes 1; \\
\Delta(l) = 1 \otimes l + l \otimes 1, \quad \forall l \in \Gamma(L); \\
\Delta(u \cdot v) = \Delta(u) \cdot \Delta(v), \quad \forall u,v \in \cU(L),
\end{gather*}
where $1 \in R$ denotes the constant function on $M$ with value $1$. See \cite{MR1815717} for details. Explicitly, we have 
\begin{equation*}
\Delta(l_1 \cdot \cdots l_n) = 1 \otimes (l_1 \cdot  \cdots l_n) +  \sum_{\substack{ i+j= n\\ i,j \in \NN}} \sum_{\sigma \in S_i^j} (l_{\sigma(1)}  \cdots \cdot l_{\sigma(i)}) \otimes (l_{\sigma(i+1)}  \cdots \cdot l_{\sigma(n)}) 
+ (l_1 \cdot \cdots l_n) \otimes 1,
\end{equation*}
for any $l_1,  \cdots, l_n \in \Gamma(L)$.

Let $(L,A)$ be a Lie pair, and $\cU(L)\Gamma(A)$ be the left ideal of $\cU(L)$ generated by $\Gamma(A)$. It is easy to check that the quotient $\frac{\cU(L)}{\cU(L) \Gamma(A)}$ is automatically an $R$-coalgebra.

Similarly, the symmetric tensors $\Gamma(SB)$ of the quotient bundle $B=L/A$ is also equipped with a standard comultiplication $\Delta:\Gamma(SB) \to \Gamma(SB) \otimes_R \Gamma(SB)$, 
$$
\Delta(b_1 \odot \cdots \odot b_n) = 1 \otimes (b_1 \odot \cdots \odot b_n) + \sum_{\substack{ i+j= n\\ i,j \in \NN}} \sum_{\sigma \in S_i^j} (b_{\sigma(1)}  \odot \cdots \odot b_{\sigma(i)}) \otimes (b_{\sigma(i+1)}  \odot \cdots \odot b_{\sigma(n)}) 
+(b_1 \odot \cdots \odot b_n) \otimes 1,
$$
for any $b_1, \cdots, b_n \in \Gamma(B)$. 

The following proposition is an extension of the classical Poincar\'{e}--Birkhoff--Witt isomorphism to Lie pairs. 

\begin{proposition}[{\cite[Theorem~2.1]{MR4271478}}]\label{prop:PBW}
Let $(L,A)$ be a Lie pair, and $B=L/A$. Given a splitting $\ib:B \to L$ of the short exact sequence 
\begin{equation}\label{eq:splitting}
\begin{tikzcd}
0 \ar[r] &  A \ar[r, hook, "\ia"'] & \ar[l, bend right, dashed,"\pa"'] L \ar[r, two heads,"\pb"'] & \ar[l, bend right, dashed,"\ib"'] B \ar[r] & 0,
\end{tikzcd}
\end{equation}
and an $L$-connection $\nabla$ on $B$, there exists a unique isomorphism of $R$-coalgebras 
$$
\pbw: \Gamma(SB) \to \frac{\cU(L)}{\cU(L) \Gamma(A)}
$$
satisfying 
\begin{align}
\pbw(f) & = f, \\
\pbw(b) & = \ib(b), \\
\pbw(b^{\odot n+1}) & = \ib(b) \cdot \pbw(b^{\odot n}) - \pbw\big(\nabla_{\ib(b)}(b^{\odot n})\big) , \label{eq:PBWiteration}
\end{align}
for any $f \in R$, $b \in \Gamma(B)$, and $n \in \NN$.
\end{proposition}

Equation~\eqref{eq:PBWiteration} is equivalent to the iteration equation 
\begin{equation}\label{eq:PBWVer2iteration}
\begin{split}
\pbw(b_0 \odot \cdots \odot b_n) = \frac{1}{n+1} \sum_{i=0}^n & \Big(\ib(b_i) \cdot \pbw(b_0 \odot \cdots \odot \widehat{b_i} \odot \cdots \odot b_n) \\
&\qquad  - \pbw\big(\nabla_{\ib(b_i)}(b_0 \odot \cdots \odot \widehat{b_i} \odot \cdots \odot b_n)\big) \Big)
\end{split}
\end{equation}
for any $b_0, \cdots, b_n \in \Gamma(B)$. 
Equation~\eqref{eq:PBWVer2iteration} indicates how one can compute the value of $\pbw$ at an arbitrary element. It also motivated the exploration of formal exponential maps for graded manifolds.  See \cite{MR3910470,MR4393962}.

When $L=T_M$ and $A$ is the trivial Lie subalgebroid of rank $0$, the $\pbw$ map is the inverse of `complete symbol map' which is an isomorphism from the space $\cU(T_M)$ of differential operators on $M$ to the space $\Gamma(S(T_M))$ of fiberwise polynomial functions on $T_M\dual$. 
The complete symbol map was generalized to arbitrary Lie algebroids by Nistor--Weinstein--Xu \cite{MR1687747}. It played an important role in quantization theory \cite{MR0706215,MR1231231,MR1427063,MR1687747}.

\subsection{Formal graded manifolds and Fedosov manifolds}

Fedosov manifolds are (formal) graded manifolds whose function algebra consists of fiberwise formal power series on $B=L/A$ with coefficients in $\Lambda^\bullet L\dual$. 
For us, a function on a graded manifold is required to be a polynomial in virtual coordinates. A graded manifold of Fedosov type will be referred to as a formal graded manifold.

Explicitly, a \textbf{($\ZZ$-)graded manifold} $\cM$ is a pair $(M,\cO_\cM)$, where $M$ is a smooth manifold, and $\cO_\cM$ is a sheaf of $\ZZ$-graded commutative $\cO_M$-algebras over $M$ such that there exist (i) a $\ZZ$-graded vector space $V$, (ii) an open cover of $M$, and (iii) an isomorphism of sheaves of graded $\cO_U$-algebras $\cO_\cM|_U \cong \cO_U \otimes SV\dual$ for every open set $U$ of the cover. We denote by $R=C^\infty(M)$ the algebra of smooth functions on the base manifold $M$, and by $\cR = C^\infty(\cM)$ the $R$-algebra of global sections of $\cO_\cM$.
See, for example, \cite{10.1093/imrn/rnae023,MR2819233,MR2709144}. 

By a \textbf{formal graded manifold} $\widehat{\cM}$, we mean a graded manifold $\cM$ together with a sheaf $\cO_\cI$ of proper homogeneous ideals of the sheaf $\cO_\cM$ of $\cO_M$-algebras. 
The {\bf function algebra} $\widehat{\cR} = C^\infty(\widehat{\cM})$ of $\widehat{\cM}$ is the $\cI$-adic completion of $\cR$, i.e.\ 
$$
\widehat{\cR} = C^\infty(\widehat{\cM}) = \varprojlim_k \cR \big/ \cI^k
$$
which is equipped with the $\cI$-adic topology. Here, $\cI$ denotes the global sections of $\cO_{\cI}$.

\begin{remark}
Since the ideal $\cI$ is required to be homogeneous, the \textbf{homogeneous elements} and the \textbf{degrees} in $\widehat{\cR}$ are automatically defined. Nevertheless, the algebra $\widehat{\cR}$ is \emph{not} necessarily a graded algebra since $\widehat{\cR}$ is not necessarily the direct sum of its homogeneous parts. Following is an example.
\end{remark}

\begin{example}
Let $V$ and $W$ be two graded vector spaces. We consider the formal graded manifold $V \oplus W_{\formal}$ arising from the ideal 
$$
\cI = S(V\dual) \otimes S^{\geq 1} (W\dual)
$$
of $C^\infty(V\oplus W) = S(V\dual) \otimes S(W\dual)$.  Its function algebra is 
$$
C^\infty(V \oplus W_{\formal}) = \varprojlim_k \frac{C^\infty(V\oplus W)}{\cI^k} =  \prod_{k=0}^\infty S (V\dual) \otimes S^{k}(W\dual) = S (V\dual)
\hotimes \widehat S(W\dual).
$$
Let $V=\RR$ and $W = \RR[2]$. 
Suppose $x$ is a nonzero vector in $V\dual$, and $y$ is a nonzero vector in $W\dual$ Then a function $f \in C^\infty(V \oplus W_{\formal})$ is of the form
$$
f = \sum_{k=0}^\infty (\sum_{i=0}^{n_k} a_i x^i) \otimes y^k,
$$
where $a_i \in \KK$. For example, one has the function $\sum_{k=0}^\infty 1 \otimes y^k$ on $V \oplus W_{\formal}$ which is \emph{not} a finite sum of homogeneous functions since each summand has a different degree: $|1 \otimes y^k| = 2k$.  
\end{example}

Let $(L,A)$ be a Lie pair, and $B=L/A$. The \textbf{Fedosov manifold} associated with $(L,A)$ is the formal graded manifold $L[1] \oplus B$ whose function algebra the $\cI$-adic completion of $\Gamma(S(L[1])\dual \otimes SB\dual) \cong \Gamma(\Lambda^\bullet L\dual \otimes SB\dual)$, i.e.\
$$
C^\infty(L[1] \oplus B) = \Gamma(\Lambda^\bullet L\dual \otimes \widehat{S}B\dual).
$$ 
Here, $\cI$ is the homogeneous ideal $\cI = \Gamma(\Lambda^\bullet L\dual \otimes S^{\geq 1} B\dual)$, and the subspace of homogeneous functions of degree $p$ in $C^\infty(L[1] \oplus B)$ is $\Gamma(\Lambda^p L\dual \otimes \widehat{S}B \dual)$.

On a formal graded manifold $\widehat{\cM}$, a vector field of degree $k$ is a derivation of degree $k$ of the function algebra $C^\infty(\widehat{\cM})$. A \textbf{(formal) dg manifold} $(\widehat\cM,Q)$ is a formal graded manifold $\widehat\cM$ endowed with a {\em homological vector field} $Q$, i.e.\ a vector field $Q$ of degree $+1$ satisfying $Q \circ Q =0$.

\subsection{Fedosov dg manifolds associated with a Lie pair}\label{sed:FedosovDGLiePair}

Let $(L,A)$ be a Lie pair over a manifold $M$, and $B=L/A$ the quotient vector bundle. 
Given a splitting $\ia \circ \pa + \ib \circ \pb = \id_L$ of the short exact sequence \eqref{eq:splitting} and a torsion-free $L$-connection $\nabla$ on $B$, one has a Fedosov dg manifold $(L[1] \oplus B, Q)$ as in \cite{MR4150934}. 

More explicitly, let $\eta^1, \cdots, \eta^r$ be a local frame of $B\dual$, and $\zeta^1, \cdots, \zeta^{r'}$ a local frame of $A\dual$. By the chosen splitting, we have an induced local frame of $L\dual$:
$$
\xi^i = \begin{cases}
\pb\dual(\eta^i) & \text{ if } i\leq r, \\
\pa\dual(\zeta^{i-r}) & \text{ if } i > r. 
\end{cases}
$$
These local frames, together with an underlying local coordinate system $x^1, \cdots, x^n$ on the base manifold $M$, induce a local coordinate system
\begin{equation}\label{eq:LocCoordFedosov}
x^1, \cdots, x^n, \xi^1, \cdots, \xi^{r+r'}, \eta^1, \cdots, \eta^r
\end{equation}
on the Fedosov manifold $\cM =L[1]\oplus B$. 
 
On $\cM$, one has the {\bf Koszul vector field} $\koszul$ which is locally of the form
\begin{equation}\label{eq:KoszulVF}
\koszul = \sum_{i=1}^r \xi^i \frac{\partial}{\partial \eta^i}.
\end{equation}
It is straightforward to show that the Koszul vector field is a globally defined vector field of degree $+1$. 
Similarly, one also has a vector field $\hat{\koszul}$ of degree $-1$ on $\cM$ which is locally of the form 
$$
\hat{\koszul} = \sum_{i=1}^r \eta^i \frac{\partial}{\partial \xi^i}.
$$
The vector field 
$$
\euler:= [\koszul,\hat\koszul] = \sum_{i=1}^r \eta^i \frac{\partial}{\partial \eta^i} + \xi^i \frac{\partial}{\partial \xi^i}
$$ 
is referred to as the {\bf Euler vector field}. As an operator on $\cR= C^\infty(\cM) = \Gamma(\Lambda^\bullet L\dual \otimes \widehat{S}B\dual)$, we have 
$$
\euler\Big( a_{I,J}\, \xi^I \eta^J  \Big) =   (|I_B|+|J|)a_{I,J}\, \xi^I \eta^J,  
$$
where $I_B$ consists of the first $r$ components of the multi-index $I$.  
By the chosen splitting, we have the isomorphism 
$$
\Lambda^p L\dual \cong \bigoplus_{p_1+p_2=p} \Lambda^{p_1} A\dual \otimes \Lambda^{p_2} B\dual.
$$ 
Let $\htp$ be the operator on $\cR $ defined by  
$$
\htp  \,\Big\vert_{\Gamma(\Lambda^{p_1} A\dual \otimes \Lambda^{p_2} B\dual \otimes S^q B\dual)} = \begin{cases}
0  & \text{ if } p_2=q=0, \\
 \dfrac{1}{p_2+q} \, \hat{\koszul} &  \text{ if } p_2+q >0.
\end{cases}
$$
This operator is a homotopy operator in the sense that 
\begin{equation}\label{eq:HtpEqForFunctions}
\koszul \htp + \htp \koszul = \id_\cR - \sigma_0,
\end{equation}
where $\sigma_0$ is the composition $\Gamma(\Lambda^\bullet L\dual \otimes \widehat{S}B\dual) \onto \Gamma(\Lambda^\bullet A\dual \otimes S^0 B\dual) \into \Gamma(\Lambda^\bullet L\dual \otimes \widehat{S}B\dual)$.
This homotopy operator extends to another operator $\htpd :=\htp \otimes \id_{B}:\Gamma(\Lambda^\bullet L\dual \otimes \widehat{S}B\dual \otimes B) \to \Gamma(\Lambda^\bullet L\dual \otimes \widehat{S}B\dual \otimes B)$ on the space of vertical vector fields.

\begin{proposition}[{\cite[Proposition~4.6]{MR4150934}}]\label{prop:FedosovDG}
Let $(L,A)$ be a Lie pair with quotient vector bundle $B = L/A$. Given a splitting of \eqref{eq:splitting} and a torsion-free $L$-connection $\nabla$ on $B$, there exists, on the formal graded manifold $L[1]\oplus B$, a unique vertical vector field 
$$
X^\nabla \in \Gamma(L\dual \otimes \widehat{S}^{\geq 2} B\dual \otimes B)
$$ 
satisfying $\htpd(X^\nabla) = 0$ and such that the vector field $Q$ of degree $+1$ on $L[1]\oplus B$ defined by
$$
Q = - \koszul + d_L^\nabla + X^\nabla.
$$
satisfies $Q^2 =0$.
\end{proposition}

The vector field $Q$ will be referred to as a {\bf Fedosov (homological) vector field}, and the formal dg manifold $(L[1] \oplus B, Q)$ is called a \textbf{Fedosov dg manifold}. 

Consider the decomposition $X^\nabla = \sum_{k=2}^\infty X_k$, where $X_k \in \Gamma(L\dual \otimes S^kB\dual \otimes B)$. The vector field $X^\nabla$ is determined by the iteration formula:
\begin{align*}
X_2 & = \htpd(R^\nabla), \\
X_{k+1} & = \htpd\Big( d_L^\nabla \circ X_k + X_k \circ d_L^\nabla + \sum_{\substack{p+q = k+1 \\ 2 \leq p,q \leq k-1}} X_p \circ X_q\Big), \qquad \forall \, k \geq 2.
\end{align*}
Here, $R^\nabla$ is the curvature of the $L$-connection $\nabla$, i.e.\ $d_L^\nabla \circ d_L^\nabla = R^\nabla$. 

There is another construction of the Fedosov homological vector field $Q$ in terms of the PBW map
$$
\pbw: \Gamma(SB) \to \frac{\cU(L)}{\cU(L) \Gamma(A)}
$$ 
from Proposition~\ref{prop:PBW}. Note that the $R$-coalgebra $\frac{\cU(L)}{\cU(L) \Gamma(A)}$ is endowed with a canonical infinitesimal $L$-action by coderivations. Pulling back this action through $\pbw$, one obtains a flat $L$-connection $\nabla^\lightning$ on $SB$:
$$
\nabla_l^\lightning(s) = \pbw\inv\big(l \cdot \pbw(s) \big),
$$
for any $l \in \Gamma(L)$ and $s \in \Gamma(SB)$. This $L$-connection $\nabla^\lightning$ induces an $L$-connection on the dual bundle $\widehat{S}B\dual$. We denote the corresponding covariant derivative by 
$$
d_L^{\nabla^\lightning}: \Gamma(\Lambda^\bullet L\dual \otimes \widehat{S}B\dual) \to \Gamma(\Lambda^{\bullet +1}L\dual \otimes \widehat{S}B\dual ).
$$

\begin{proposition}[{\cite[Theorem~4.7]{MR4150934}}]\label{prop:PBWconstructionOfFedosov}
Let $(L,A)$ be a Lie pair, and $B = L/A$. Given a splitting of \eqref{eq:splitting} and a torsion-free $L$-connection $\nabla$ on $B$, the Fedosov vector field $Q$ coincides with the covariant derivative $d_L^{\nabla^\lightning}$:
$$
Q = d_L^{\nabla^\lightning}.
$$
\end{proposition}

One can find a graded-manifold version of Proposition~\ref{prop:PBWconstructionOfFedosov} in \cite{MR3910470}.

\section{Koszul-type homotopy equation for $\cI$-adic filtration-shifting vertical operators}

The main purpose of the present paper is to prove that, given two choices of a splitting of \eqref{eq:splitting} and a torsion-free $L$-connection on $B$, there exists a {\em unique} vertical isomorphism $\phi:(\cR,Q_2) \to (\cR,Q_1)$ between the Fedosov dg manifolds associated with the two choices (Theorem~\ref{thm:PhiIteration}). Here, $\cR = C^\infty(L[1]\oplus B)$ is the function algebra of a Fedosov manifold $L[1] \oplus B$, and $Q_1$ and $Q_2$ are the Fedosov vector fields associated with the two choices. Our method is based on an iteration equation for vertical isomorphisms of Fedosov dg manifolds:
\begin{equation}\label{eq:Sec2PhiIteration}
\phi =  1+ \htpd \big((Q_1-Q_2) \circ \phi\big)+ \htpd ([\koszul + Q_2 ,\phi]).
\end{equation}
The main technical difficulty here is the existence of a Koszul-type homotopy operator $\htpd$ for $\cI$-adic filtration-shifting vertical operators.

In this section, we introduce the operator $\htpd$ and prove that it satisfies a Koszul-type homotopy equation for any $\cI$-adic filtration-shifting vertical operators (Corollary~\ref{cor:KoszulBddVO}). The  proof of Equation~\eqref{eq:Sec2PhiIteration} is postponed to Section~\ref{sec:VerIsoFedosovDG}.

\subsection{Topology on the function algebra of a Fedosov manifold}\label{sec:I-adic}

Let $(L,A)$ be a Lie pair, and $B =L/A$. The Fedosov manifold $L[1]\oplus B$ has the function algebra 
$$
\cR = \Gamma(\Lambda^\bullet L\dual \otimes \widehat{S} B\dual) = \prod_{k=0}^\infty \Gamma(\Lambda^\bullet L\dual \otimes S^k B\dual).
$$ 
Let $\cI=\Gamma(\Lambda^\bullet L\dual \otimes \widehat{S}^{\geq 1} B\dual)$. The \textbf{$\cI$-adic metric} $d_\cI:\cR \times \cR \to \RR$ is defined by 
$$
d_\cI(f,g) = \begin{cases}
\dfrac{1}{k+1} & \text{ if } f-g \in \cI^k \setminus \cI^{k+1}, \\
0 & \text{ if } f-g \in \cI^k, \; \forall k \geq 0.
\end{cases}
$$
Here we define $\cI^k$, for $k \leq 0$, to be $\cR$ for convenience. 
It is straightforward to show that the function algebra $\cR$ together with the metric $d_\cI$ is a complete metric space. Furthermore, $\cR$ is a topological ring whose topology is generated by the following neighborhoods of zero: $\cI^k = \Gamma(\Lambda^\bullet L\dual \otimes \widehat{S}^{\geq k} B\dual)$, $k\geq 0$. See, for example, \cite{MR0282956}.

The following lemma is a criterion of continuity.

\begin{lemma}\label{lem:ContiCriterion}
Let $\phi:{\cR} \to {\cR}$ be a linear map. 
If there exists a fixed integer $N \in \ZZ$ such that $\phi(\cI^k) \subset \cI^{k+N}$ for any $k \geq 0$, then $\phi$ is continuous.
\end{lemma}
\begin{proof}
Since $\cI^k$, $k \geq 0$, generate the topology of $\cR$, it suffices to show that $\phi\inv(\cI^k)$ is open for each $k$. It follows from the assumption that $\cI^{k-N} \subset \phi\inv(\cI^k)$. Thus, $\phi\inv(\cI^k) = \bigcup_{f \in \phi\inv(\cI^k)} (f + \cI^{k-N})$ is open. This proves the lemma.
\end{proof}

\begin{remark}\label{rmk:Conv&Conti}
Here is a trio of useful observations:
\begin{enumerate}
\item
Convergence criterion for series in ${\cR}$: If $a_k \in \cI^k$ for all $k \geq 0$, then the series $\sum_{k=0}^\infty a_k$ converges in ${\cR}$.

\item
If $\phi:{\cR}\to{\cR}$ is a continuous linear map and $\sum_{k=0}^\infty f_k$ is a series in ${\cR}$ with $f_k\in\Gamma(\Lambda^\bullet L^\vee\otimes S^k B^\vee)$ for all $k$, then the series $\sum_{k=0}^\infty \phi(f_k)$ converges to $\phi(\sum_{k=0}^\infty f_k)$ in ${\cR}$. Thus, a continuous linear map $\phi:{\cR}\to{\cR}$ is uniquely determined by its restrictions to the subspaces $\Gamma(\Lambda^\bullet L^\vee\otimes S^k B^\vee)$ for $k\geq0$.

\item
Convergence criterion for series of operators on ${\cR}$: If $\phi_n:\cR \to \cR$, $n \geq 0$, is a sequence of operators such that $\phi_n(\cI^k) \subset \cI^{n+k}$, for all $k\geq 0$ and $n\geq 0$, then the series $\sum_{n=0}^\infty \phi_n$ converges uniformly to a continuous linear operator $\phi:\cR \to \cR$.
\end{enumerate}
\end{remark}

\begin{example}\label{ex:VerticalDiffOp}
An element $D \in \Gamma(\Lambda^\bullet L\dual \otimes \widehat{S} B\dual \otimes SB)$ can be identified with a linear map $D:\cR \to \cR$ in the following way: 
if 
$$
D = \omega \otimes f \otimes (b_1 \odot \cdots \odot b_q) \in \Gamma(\Lambda^\bullet L\dual \otimes \widehat{S} B\dual \otimes S^q B),  
$$
then 
\begin{equation}\label{eq:VerDiffOpDef}
D(\omega' \otimes f') = (\omega \wedge \omega') \otimes (f \cdot \iota_{b_1} \cdots \iota_{b_q}(f')),
\end{equation}
where $\omega, \omega' \in \Gamma(\Lambda^\bullet L\dual)$, $f, f' \in \Gamma(\widehat{S}B\dual)$, and $b_1, \cdots, b_k \in \Gamma(B)$. For $b \in \Gamma(B)$, the {\bf contraction} $\iota_b:\cR \to \cR$ is the $\Gamma(\Lambda^\bullet L\dual)$-linear continuous derivation of degree zero satisfying $\iota_b(1 \otimes \eta) = \pair{b}{\eta}$ for any $\eta \in \Gamma(B\dual)$. 

Observe that if $D \in \Gamma(\Lambda^\bullet L\dual \otimes \widehat{S} B\dual \otimes S^q B)$, then $D(\cI^k)\subset\cI^{k-q}$ for all $k$. Thus, by Lemma~\ref{lem:ContiCriterion}, every element of $\Gamma(\Lambda^\bullet L\dual \otimes \widehat{S} B\dual \otimes SB)$ is a continuous linear operator on $\cR$.
\end{example}

\subsection{Vertical differential operators}

In Example~\ref{ex:VerticalDiffOp}, we introduced an important class of continuous linear operators on $\cR$ which will be called {\em vertical differential operators}, meaning formal fiberwise differential operators on the vector bundle $ L[1]\oplus B \to L[1]$ tangent to the fiber-directions. Following is a theoretical definition of vertical differential operators. 

\begin{definition}
Let $(L,A)$ be a Lie pair, and $\cM = L[1] \oplus B$ be the associated Fedosov manifold. 
A \textbf{vertical vector field} on $\cM$ is a continuous derivation of the function algebra $\cR = \Gamma(\Lambda^\bullet L\dual \otimes \widehat{S}B\dual)$ which is linear over $\Gamma(\Lambda^\bullet L\dual)$. Let $\vdop$ denote the subalgebra of $\End_{\KK}(\cR)$ generated by the vertical vector fields and the multiplications by function $f \in \cR$. An element in $\vdop$ is called a \textbf{vertical differential operator} on $\cM$.
\end{definition}

Let $X$ be a vertical vector field on $\cM$. Due to the continuity and $\Gamma(\Lambda^\bullet L\dual)$-linearity of $X$, it is uniquely determined by its restrictions to $\Gamma(S^k B\dual) = \Gamma(\Lambda^0 L\dual \otimes S^k B\dual)$, $k \geq 0$. Furthermore, each restriction is determined by the derivation law and the restriction to $k=1$. Consequently, the space of vertical vector fields is isomorphic to $\Gamma(\Lambda^\bullet L\dual \otimes \widehat{S}B\dual \otimes B)$ by the formula \eqref{eq:VerDiffOpDef} with $q =1$. 
Since $\Gamma(\Lambda^\bullet L\dual \otimes \widehat{S}B\dual \otimes SB)$ is generated by $\Gamma(\Lambda^\bullet L\dual \otimes \widehat{S}B\dual \otimes B)$ and $\cR = \Gamma(\Lambda^\bullet L\dual \otimes \widehat{S}B\dual )$, it follows that Equation~\eqref{eq:VerDiffOpDef} defines an isomorphism of graded algebras
\begin{equation}\label{eq:vdop}
\vdop \cong \Gamma(\Lambda^\bullet L\dual \otimes \widehat{S}B\dual \otimes SB)
\end{equation}
which is also a graded Lie algebra with respect to the bracket $[\argument,\argument]$ of graded commutators.

Let $\koszul$ be the Koszul vector field defined as in \eqref{eq:KoszulVF}. Since $\koszul$ is a degree-one vertical vector field with $[\koszul,\koszul]=0$, the operator
$$
\delta:=[\koszul,\argument]:\vdop \to \vdop
$$
is a coboundary operator: $\delta^2 =0$.

\begin{lemma}
With the identification \eqref{eq:vdop}, the operator $\delta$ coincides with $\koszul \otimes \id_{SB}$:
$$
\delta = \koszul \otimes \id_{SB}: \Gamma(\Lambda^\bullet L\dual \otimes \widehat{S} B\dual \otimes SB) \to  \Gamma(\Lambda^\bullet L\dual \otimes \widehat{S} B\dual \otimes SB).
$$
\end{lemma}
\begin{proof}
It suffices to show that $\delta(1 \otimes 1 \otimes s)=0$ for any $s \in \Gamma(SB)$. 
Following the notations of local coordinates in \eqref{eq:LocCoordFedosov}, 
$$
\delta\Big(\frac{\partial}{\partial \eta^{i_1}} \cdot \cdots \frac{\partial}{\partial \eta^{i_q}}\Big) = \sum_{j=1}^q  \frac{\partial}{\partial \eta^{i_1}} \cdots \Big[\koszul, \frac{\partial}{\partial \eta^{i_j}}\Big] \cdots \frac{\partial}{\partial \eta^{i_n}}
$$
which vanishes since 
$$
\Big[\koszul, \frac{\partial}{\partial \eta^{i_j}}\Big] = -\sum_{k=1}^{\rk B} \frac{\partial}{\partial \eta^{i_j}}(\xi^k)\frac{\partial}{\partial \eta^{k}} =0. 
$$
This completes the proof.
\end{proof}

Let
$$
\htpd:=\htp \otimes \id_{SB}: \Gamma(\Lambda^\bullet L\dual \otimes \widehat{S} B\dual \otimes SB) \to \Gamma(\Lambda^\bullet L\dual \otimes \widehat{S} B\dual \otimes SB).$$ It follows from Equation~\eqref{eq:HtpEqForFunctions} that, for any $D \in  \vdop \cong \Gamma(\Lambda^\bullet L\dual \otimes \widehat{S} B\dual \otimes SB)$, 
\begin{equation}\label{eq:HtpEqForDiffOp}
\delta \htpd (D)+\htpd\delta(D) = D - \sigma_0(D),
\end{equation}
where $\sigma_0$ is the projection operator
\begin{multline*}
\sigma_0:\Gamma(\Lambda^\bullet L\dual \otimes \widehat{S} B\dual \otimes SB)\cong \Gamma(\Lambda^\bullet A\dual \otimes \Lambda^\bullet B\dual \otimes \widehat{S} B\dual \otimes SB)  \\
\onto \Gamma(\Lambda^\bullet A\dual \otimes \Lambda^0 B\dual \otimes  S^0 B\dual \otimes SB)\cong \Gamma(\Lambda^\bullet A\dual  \otimes SB)  \into \Gamma(\Lambda^\bullet L\dual \otimes \widehat{S} B\dual \otimes SB).
\end{multline*}

The following lemma is straightforward.

\begin{lemma}\label{lem:HomotopyEq}
As operators on $\vdop$, 
\begin{gather*}
\delta^2=0, \qquad \htpd^2=0, \qquad \htpd \delta \htpd = \htpd,\\
[\delta,\htpd] = \delta \htpd + \htpd \delta = \id -\sigma_0.
\end{gather*}
\end{lemma}

Lemma~\ref{lem:HomotopyEq} guarantees the existence of a contraction with differential $\delta$ and homotopy operator $\htpd$. See \cite[Appendix~A]{MR4665716}. 
This contraction is referred to as the {\em Koszul contraction for vertical differential operators}.

\subsection{$\cI$-adic filtration-shifting vertical operators}\label{sec:BddBelowVerOp}

Now we consider a wider class of continuous linear operators on $\cR=C^\infty(L[1] \oplus B) = \Gamma(\Lambda^\bullet L\dual \otimes \widehat{S} B\dual)$. 

\begin{definition}
A continuous linear map $\phi:\cR \to \cR$ is called a \textbf{vertical operator} on $\cR$ if it is  $\Gamma(\Lambda^\bullet L\dual)$-linear. Let $N$ be an integer. A vertical operator $\phi:\cR \to \cR$ is said to be an \textbf{$N$-shifting} operator (or \textbf{shift the $\cI$-adic filtration by $N$}) if  
\begin{equation}\label{eq:DegPropertyEndo}
\phi\big(\Gamma(\Lambda^\bullet L\dual \otimes S^k B\dual)\big) \subset \Gamma(\Lambda^\bullet L\dual \otimes \widehat{S}^{\geq k+N} B\dual), \qquad \forall \, k \geq 0.
\end{equation}
We denote by $\sE^{\geq N}$ the space of $N$-shifting vertical operators. The algebra of {\bf $\cI$-adic filtration-shifting} vertical operators on $\cR$ is 
$$
\sE = \bigcup_{N \in \ZZ} \sE^{\geq N}.
$$
\end{definition}

\begin{example}
Let $q \in \NN$ and $N \in \ZZ$. 
A vertical differential operator in $\Gamma(\Lambda^\bullet L\dual \otimes \widehat{S}^{\geq q+N} B\dual \otimes S^qB)$ is an $N$-shifting operator. 
\end{example}

In fact, any $\cI$-adic filtration-shifting vertical operator is a uniform limit of vertical differential operators. 

\begin{lemma}\label{lem:UniformConvDq}
Fix an integer $N$. Given any sequence of vertical differential operators $D_q \in \Gamma(\Lambda^\bullet L\dual \otimes \widehat{S}^{\geq q+N} B\dual \otimes S^q B)$, $q\geq0$, on $\cM=L[1]\oplus B$, the series $\sum_{q=0}^\infty D_q$ converges uniformly to an $N$-shifting vertical operator.
\end{lemma}
\begin{proof} 
Consider the algebra $\whole$ with the $\cI$-adic topology generated by its ideal $\cI=\tail{1}$.

Let $D_0, D_1, D_2, \cdots$ be a sequence of ($N$-shifting) vertical differential operators on the algebra $\whole$ with
\[ D_q\in\Gamma(\Lambda^\bullet L\dual\otimes \widehat{S}^{\geq q+N}B\dual\otimes S^q B),
\qquad\forall q\in \NN_0 .\]
Note that
\begin{equation}\label{one}
D_q\big(\tail{n}\big)\subset\tail{n+N},
\qquad\forall q\in \NN_0
.\end{equation}
In other words, each operator $D_q$ preserves the $\cI$-adic filtration
on $\whole$ up to an $N$-shift.

{We claim that the series of operators $\sum_{q=0}^\infty D_q$
converges uniformly.}

Recall that $D_q(f_k)=0$ if $f_k\in\homogeneous{k}$ and $q>k$.
In other words, we have
\begin{equation}\label{two}
\head{q}\subset\ker(D_q)
.\end{equation}

It follows from \eqref{two} and \eqref{one} that
\[ \im(D_q)=D_q\big(\tail{q}\big)\subset\tail{q+N},
\qquad \forall q\in\NN_0 .\]

Therefore the series $\sum_{q=0}^\infty D_q$ converges, for only finitely many of the $D_q$ contribute to any one of the homogeneous components of $\whole$: given any $l\in\NN_0$, we have $\pr_{\homogeneous{l}}\circ D_q=0$ for all but finitely many values of the index $q$.

Furthermore, it follows from \eqref{two} that
\[ \head{n}\subset\head{n+r}\subset\ker(D_{n+r}),
\qquad \forall r\in\NN_0 ,\]
and, consequently,
\begin{equation}\label{three}
\head{n}\subset\bigcap_{q=n}^\infty\ker(D_q)\subset\ker\big(\sum_{q=n}^\infty D_q\big)
.\end{equation}
On the other hand, it follows from \eqref{one} that
\begin{equation}\label{four}
\sum_{q=n}^\infty D_q\big(\tail{n})\subset\tail{n+N}
.\end{equation}

It follows from \eqref{three} and \eqref{four} that
\[ \im\big(\sum_{q=n}^\infty D_q\big)
=\sum_{q=n}^\infty D_q\big(\tail{n}\big)
\subset\tail{n+N} ,\]
i.e.\ 
\begin{equation}\label{five}
\im\big(\sum_{q=n}^\infty D_q\big)
\subset\cI^{n+N}
.\end{equation}

Since the $n$-tail $\sum_{q=n}^\infty D_q$ of the series
takes values in the $(n+N)$-th power
of the ideal generating the $\cI$-adic topology,
the series of operators $\sum_{q=0}^\infty D_q$
converges uniformly.
\end{proof}

\begin{proposition}\label{prop:BddOp&DiffOp}
Let $N$ be an integer. A vertical operator $\phi$ on $\cR$ shifts the $\cI$-adic filtration by $N$ if and only if there exists a unique sequence of vertical differential operators $D_q \in \Gamma(\Lambda^\bullet L\dual \otimes \widehat{S}^{\geq q+N} B\dual \otimes S^q B)$, $q \geq 0$, such that the series $\sum_{q=0}^\infty D_q$ converges uniformly to $\phi$.
\end{proposition}
\begin{proof}
It remains to show that a vertical operator $\phi$ satisfying \eqref{eq:DegPropertyEndo} can be uniquely written as $\phi = \sum_{q=0}^\infty D_q$. 
Suppose $\phi:\cR \to \cR$ is an $N$-shifting vertical operator. Let $\phi^q$ be the $R$-linear map
$$
\phi^q:\Gamma(S^q B\dual) = \Gamma(\Lambda^0 L\dual \otimes S^q B\dual) \into \Gamma(\Lambda^\bullet L\dual \otimes \widehat{S} B\dual ) \xto{\phi} \Gamma(\Lambda^\bullet L\dual \otimes \widehat{S}^{\geq q+N} B\dual ),
$$
where $R = C^\infty(M)$. 

Observe that the elements in $\Gamma(\Lambda^\bullet L\dual \otimes \widehat{S}^{\geq q+N} B\dual \otimes S^q B)$ are in one-to-one correspondence to the $R$-linear maps $\Gamma(S^q B\dual)  \to \Gamma(\Lambda^\bullet L\dual \otimes \widehat{S}^{\geq q+N} B\dual )$, where the correspondence is given by restricting the vertical differential operator to $\Gamma(S^q B\dual)$. 

Now we choose $D_q \in \Gamma(\Lambda^\bullet L\dual \otimes \widehat{S}^{\geq q+N} B\dual \otimes S^q B)$ inductively:
\begin{align}
D_0 \big|_{\Gamma(S^0B\dual)}& = \phi^0, \label{eq:D_0inDecomposition} \\
D_{q} \big|_{\Gamma(S^{q}B\dual)}& = \phi^q-\sum_{i=0}^{q-1} D_{i} \big|_{\Gamma(S^{q}B\dual)}.
\end{align} 
By Lemma~\ref{lem:UniformConvDq}, the series $\sum_{q=0}^\infty D_q$ converges uniformly to a vertical operator. By the construction, 
\begin{equation}\label{eq:BddOp&DiffOp}
\phi \big|_{\Gamma(S^q B\dual)} = \phi^q = \sum_{i=0}^q D_i \big|_{\Gamma(S^q B\dual)} = \sum_{i=0}^\infty D_i \big|_{\Gamma(S^q B\dual)}, \qquad \forall \, q \geq 0.
\end{equation}
Thus, we have $\phi = \sum_{q=0}^\infty D_q$. 

Finally, if $\phi = \sum_{q=0}^\infty D_q$, then by restricting $\phi$ to each $\Gamma(S^q B\dual)$, we obtain Equation~\eqref{eq:BddOp&DiffOp} which uniquely determines $D_q$, $q \geq 0$. 
\end{proof}

As a consequence, we can describe $\sE^{\geq N}$ in the following way.

\begin{corollary}
For each $N \in \ZZ$, the map
$$
 \prod_{q=0}^\infty \Gamma(\Lambda^\bullet L\dual \otimes \widehat{S}^{\geq q+N} B\dual \otimes S^q B) \to \sE^{\geq N}, \qquad (D_q)_{q=0}^\infty \mapsto \sum_{q=0}^\infty D_q
$$
is an isomorphism of $R$-modules. 
\end{corollary}

\subsubsection*{Koszul-type homotopy equation for $\cI$-adic filtration-shifting vertical operators}

Since 
\begin{align*}
\delta : & \Gamma(\Lambda^p L\dual \otimes S^q B\dual \otimes SB) \to \Gamma(\Lambda^{p+1} L\dual \otimes S^{q-1} B\dual \otimes SB), \\
\htpd: & \Gamma(\Lambda^p L\dual \otimes S^q B\dual \otimes SB) \to \Gamma(\Lambda^{p-1} L\dual \otimes S^{q+1} B\dual \otimes SB),
\end{align*}
we have that $\delta(\sE^{\geq N}) \subset \sE^{\geq N-1}$ and $\htpd(\sE^{\geq N}) \subset \sE^{\geq N+1}$. 
By Lemma~\ref{lem:HomotopyEq} and Proposition~\ref{prop:BddOp&DiffOp}, we have the following 

\begin{corollary}\label{cor:KoszulBddVO}
The operators 
\begin{align*}
\delta:\sE \to \sE, & \qquad \delta\Big(\sum_{q=0}^\infty D_q\Big) := \sum_{q=0}^\infty \delta(D_q), \\
\htpd:\sE \to \sE, &\qquad \htpd\Big(\sum_{q=0}^\infty D_q\Big) := \sum_{q=0}^\infty \htpd(D_q), \\
\sigma_0:\sE \to \sE, &\qquad \sigma_0\Big(\sum_{q=0}^\infty D_q\Big) := \sum_{q=0}^\infty \sigma_0(D_q)
\end{align*}
satisfy the equations 
\begin{gather*}
\delta^2=0, \qquad \htpd^2=0, \qquad \htpd \delta \htpd = \htpd,\\
[\delta,\htpd] = \delta \htpd + \htpd \delta = \id_\sE -\sigma_0.
\end{gather*}
\end{corollary}

\section{Vertical isomorphisms of Fedosov dg manifolds}

Let $\cM = L[1]\oplus B$ be the Fedosov manifold associated with a Lie pair $(L,A)$, where $B = L/A$. We denote the function algebra of $\cM$ by $\cR = C^\infty(\cM) = \Gamma(\Lambda^\bullet L\dual \otimes \widehat{S}B\dual)$. 
Given two splittings $\ib_1$, $\ib_2$ of the short exact sequence \eqref{eq:splitting} and given torsion-free $L$-connections $\nabla_1$ and $\nabla_2$ on $B$, one has two Fedosov homological vector fields from Proposition~\ref{prop:FedosovDG}:
\begin{align}
Q_1 & = -\koszul + d_L^{\nabla_1} + X_1, \label{eq:Q_1} \\
Q_2 & = -\koszul + d_L^{\nabla_2} +  X_2 \label{eq:Q_2}
\end{align}
on $\cM$. 

\begin{definition}
A vertical operator on $\cR$ is called a \textbf{vertical isomorphism of Fedosov dg manifolds} from $(\cM,Q_1)$ to $(\cM, Q_2)$ if it is an isomorphism of dg algebras from $(\cR,Q_2)$ to $(\cR,Q_1)$. 
\end{definition}

It can be shown that any vertical isomorphism $\phi:(\cR,Q_2) \to (\cR,Q_1)$ preserves the $\cI$-adic filtration, i.e.\ $\phi \in \sE^{\geq 0} \subset \sE$. See Remark~\ref{rmk:VerAlgIsoBddBelow}.

Now we are ready to state the main theorem:

\begin{theorem}\label{thm:PhiIteration}
Let $(L,A)$ be a Lie pair with quotient $B=L/A$.
Given two Fedosov vector fields $Q_1$ and $Q_2$
on the Fedosov manifold $\cM=L[1]\oplus B$
(arising from two choices of splittings of the short exact sequence \eqref{eq:splitting}
and torsion-free $L$-connections on $B$),
there exists a unique vertical automorphism $\phi$ of the graded algebra $\cR=C^\infty(\cM)$
intertwining the two Fedosov vector fields $Q_1$ and $Q_2$,
i.e.\ satisfying $\phi\circ Q_2=Q_1\circ\phi$.
This vertical automorphism $\phi$ is the unique solution of the equation
\begin{equation}\label{eq:PhiIteration} \phi=1+ \htpd\eth(\phi) ,\end{equation}
where $\eth(\phi)=\Delta Q\circ\phi+[\koszul+Q_2,\phi]$ and $\Delta Q=Q_1-Q_2$.
Indeed, we have
\begin{equation}\label{eq:PhiSeries} \phi=\sum_{n=0}^\infty (\htpd\eth)^n(1)
=1+\sum_{n=0}^\infty (\htpd\eth)^n \big(\htpd(\Delta Q)\big)
.\end{equation}
\end{theorem}

Theorem~\ref{thm:PhiIteration} will be proved in Section~\ref{sec:VerIsoFedosovDG}.

\subsection{Vertical automorphisms of a Fedosov manifold}

In this subsection, we study the underlying algebra isomorphism $\phi:\cR \to \cR$ of a vertical isomorphism of Fedosov dg manifolds. 

\begin{definition}
A \textbf{vertical automorphism} of the Fedosov manifold $\cM = L[1] \oplus B$ is a continuous $\Gamma(\Lambda^\bullet L\dual)$-linear degree-preserving algebra isomorphism of the graded algebra $\cR = \Gamma(\Lambda^\bullet L\dual \otimes \widehat{S} B\dual) $.
\end{definition}

\begin{example}
If $\chi:B \to B$ is an isomorphism of vector bundles, then its dual map $\chi\dual:B\dual \to B\dual$ generates a vertical automorphism $\chi^\natural: \Gamma(\Lambda^\bullet L\dual \otimes \widehat{S} B\dual) \to \Gamma(\Lambda^\bullet L\dual \otimes \widehat{S} B\dual)$, 
$$
\chi^\natural(\omega \otimes (\beta_1 \odot \cdots \odot \beta_k)) = \omega \otimes (\chi\dual(\beta_1) \odot \cdots \odot \chi\dual(\beta_k)),
$$
where $\omega \in \Gamma(\Lambda^\bullet L\dual)$, $\beta_1, \cdots, \beta_k \in \Gamma(B\dual)$, and $k \geq 1$. The continuity of $\chi^\natural$ is guaranteed by Lemma~\ref{lem:ContiCriterion}.
\end{example}

By the continuity assumption, we have that, for any $f_k \in \Gamma(\Lambda^\bullet L\dual \otimes S^{k} B\dual)$, the sum $\sum_{k=0}^\infty \phi(f_k)$ converges to $\phi\big(\sum_{k=0}^\infty f_k\big)$ in $\cR$. See Remark~\ref{rmk:Conv&Conti}. Furthermore, by the algebra morphism property, each component $\phi(f_k)$ is determined by its restriction to $\Gamma(B\dual) = \Gamma(\Lambda^0 L\dual \otimes S^1 B\dual)$, and hence so is $\phi$. The restriction $\phi\big|_{\Gamma(B\dual)}$ has the decomposition
\begin{equation}\label{eq:PhiDecomposition}
\phi\big|_{\Gamma(B\dual)} = \sum_{k=0}^\infty \phi_k, 
\end{equation}
where $\phi_k$ is the composition 
$$
\begin{tikzcd}
\phi_k:\Gamma(B\dual) =\Gamma(\Lambda^0 L\dual \otimes S^1 B\dual) \ar[r,"\phi"]& \Gamma(\Lambda^\bullet L\dual \otimes \widehat{S} B\dual) \ar[r,two heads,"\pr_k"] &\Gamma(\Lambda^\bullet L\dual \otimes S^{k} B\dual).
\end{tikzcd}
$$

\begin{remark}\label{rmk:NoLInPhi}
Since $\phi$ preserves the grading, we have that 
$$
\phi_k:\Gamma(B\dual) \to \Gamma(S^k B\dual) = \Gamma(\Lambda^0 L\dual \otimes S^k B\dual).
$$
\end{remark}

The map $\phi_k:\Gamma(B\dual) \to \Gamma(S^k B\dual)$ will be called the \textbf{$k$-th component} of the vertical automorphism $\phi$.
The first component of $\phi$ is also referred to as the \textbf{linear} component. We say the linear component of $\phi$ is \textbf{trivial} if  $\phi_1 = \id_{B\dual}$.

The following two lemmas describe the zeroth and the first components of $\phi$.

\begin{lemma}\label{lem:ZeroCompPhi}
If $\phi$ is a vertical automorphism, then $\phi_0 =0$. 
\end{lemma}
\begin{proof}
If $\phi_0 \neq 0$, then there exists a nonzero $\beta \in \Gamma(B\dual)$ such that $\phi_0(\beta) \neq 0$ in $C^\infty(M) \cong \Gamma(S^0 B\dual)$. Assume $x_0 \in M$ is a point such that $c = \phi_0(\beta)(x_0) \neq 0$. Then, by the continuity of $\phi$, 
\begin{align*}
\phi\Big(\sum_{n=1}^\infty (\frac{\beta}{c})^n\Big) & = \sum_{n=1}^\infty (\phi(\frac{\beta}{c}))^n \\
& = \sum_{n=1}^\infty \big(\phi_0(\frac{\beta}{c}) +(\phi(\frac{\beta}{c})-\phi_0(\frac{\beta}{c}))\big)^n \\
& = \Big(\sum_{n=1}^\infty (\phi_0(\frac{\beta}{c}))^n \Big)+ \Big( (\phi(\frac{\beta}{c})-\phi_0(\frac{\beta}{c}) ) \cdot \Theta\Big),
\end{align*}
for some $\Theta \in \Gamma(\widehat{S}B\dual)$. Since $\phi(\frac{\beta}{c})-\phi_0(\frac{\beta}{c}) \in \Gamma(\widehat{S}^{\geq 1} B\dual)$, by projecting to $\Gamma(S^0 B\dual) \cong C^\infty(M)$, we get a well-defined smooth function $f =\sum_{n=1}^\infty (\phi_0(\frac{\beta}{c}))^n \in C^\infty(M)$. This is a contradiction because the value of $f$ at $x_0$ is the divergent series $\sum_{n=1}^\infty 1^n$.
\end{proof}

\begin{remark}\label{rmk:VerAlgIsoBddBelow}
Due to the algebra morphism property and Lemma~\ref{lem:ZeroCompPhi}, any vertical isomorphism preserves the $\cI$-adic filtration.
\end{remark}

\begin{lemma}\label{lem:1stComponentPhi}
For any vertical automorphism $\phi$, there exists a vector bundle isomorphism $\chi:B \to B$ such that $(\chi^\natural \circ \phi  )_1 = \id_{B\dual}$.
\end{lemma}
\begin{proof}
Given a vertical automorphism $\phi$, the first component $\phi_1:\Gamma(B\dual) \to \Gamma(B\dual)$ is an isomorphism of $C^\infty(M)$-modules, and its dual map is naturally identified with a vector bundle isomorphism $\chi:B \to B$. Since the given automorphism $\phi$ and the vertical automorphism $\chi^\natural$ induced by $\chi$ have the same first component, the first component of $\big((\chi^\natural)\inv \circ \phi \big)$ is $\id_{B\dual}$. Replacing $\chi$ by its inverse, we prove the lemma.
\end{proof}

Now we are ready to describe all the vertical automorphisms.

\begin{lemma}\label{lem:VerAuto}
Given any vertical automorphism $\phi$,
there exists a unique pair $(\chi,Y)$ consisting of an automorphism
$\chi$ of the vector bundle $B$ and a vertical vector field
$Y\in\Gamma(\widehat{S}^{\geq 2}B^\vee\otimes B)$ such that
\[ \phi=\chi^\natural\circ e^Y=\chi^\natural
\sum_{n=0}^\infty\frac{1}{n!}Y^n
=\sum_{n=0}^\infty\frac{1}{n!}\chi^\natural\circ
\overset{\text{$n$ times}}{\overbrace{Y\circ\cdots\circ Y}} .\]
\end{lemma}

\begin{proof}
By Lemma~\ref{lem:1stComponentPhi}, it suffices to prove that, if $\phi_1=\id_{B^\vee}$,
there exists a unique vertical vector field
$Y\in\Gamma(\widehat{S}^{\geq 2}B^\vee\otimes B)$ such that $\phi=e^Y$.

Suppose such a derivation $Y$ of the algebra
$\Gamma(\widehat{S}B^\vee)$ exists. Let $Y_q\in \Gamma(S^q B\dual \otimes B)$, $q \geq 2$, be the vector fields such that $Y=\sum_{q=2}^\infty Y_q$. 
Since $Y_q$ sends $\Gamma(S^k B^\vee)$ to $\Gamma(S^{k+q-1}B^\vee)$,
a direct computation shows that
\begin{equation}\label{twenty-four} \phi_2=Y_2\circ j \end{equation}
and, for all $q\geq 3$,
\[ \phi_q=\pr_q\circ \phi\circ j=\pr_q\circ e^Y\circ j
=\sum_{k=1}^{q-1}\frac{1}{k!}\sum_{\substack{i_1+\cdots+i_k=q-1+k \\
2\leq i_1,\cdots,i_k\leq q+1-k}}(Y_{i_1}\circ\cdots\circ Y_{i_k})\circ j ,\]
where $j$ denotes the canonical inclusion of $\Gamma(B^\vee)$
in $\Gamma(\widehat{S}(B^\vee))$.
The last equation can be rewritten as
\begin{equation}\label{eq:log_iteration}
Y_q\circ j=\phi_q-\sum_{k=2}^{q-1}\frac{1}{k!}
\sum_{\substack{i_1+\cdots+i_k=q-1+k \\
2\leq i_1,\cdots,i_k\leq q+1-k}}
(Y_{i_1}\circ\cdots\circ Y_{i_k})\circ j,
\end{equation}
which shows that, for all $q\geq 3$, the component $Y_q$ is determined by $\phi_q$ and the components $Y_2,Y_3,\dots,Y_{q-1}$
and hence, by induction on $q$, merely by $\phi$ and the component $Y_2$,
and thus ultimately by $\phi$ through Equation~\eqref{twenty-four}.
The uniqueness claim is established.

We now proceed with the proof of the existence claim.
Suppose $\phi$ is an arbitrary vertical automorphism such that $\phi_1=\id_{B^\vee}$.
Equations~\eqref{twenty-four} and~\eqref{eq:log_iteration} together
define a vertical vector field
$Y=\sum_{q=2}^\infty Y_q\in\Gamma(\widehat{S}^{\geq 2}B^\vee\otimes B)$.
Since $Y^n(\cI^k)\subset\cI^{k+n}$, the series of operators
$\sum_{n=0}^\infty\frac{1}{n!}Y^n$ converges uniformly ---
according to Remark~\ref{rmk:Conv&Conti}~(3) --- to a vertical automorphism $e^Y$
satisfying the property $\pr_1\circ e^Y\circ j=\id$.
By the construction of $Y$, we have $\phi=e^Y$ as desired.
\end{proof}

\begin{remark}
The formula of $Y$ also can be derived by the power series of $\log$:
$$
Y = \log\big((\chi^\natural)\inv\circ \phi\big) = \sum_{n=1}^\infty \frac{(-1)^{n+1}}{n} \cdot \big((\chi^\natural)\inv \circ \phi - 1\big)^n.
$$
\end{remark}

\begin{remark}
The vector bundle isomorphism $\chi$ in Lemma~\ref{lem:VerAuto} is essentially the linear component of the vertical automorphism $\phi$. We separate the linear component since we can get the iteration equation \eqref{eq:log_iteration} to determine $Y$ in this way. Also see Remark~\ref{rmk:ExpY}~(1) for another technical difficulty with the linear component. 
\end{remark}

\begin{remark}\label{rmk:ExpY}
Lemma~\ref{lem:VerAuto} is a fiberwise result. One can consider the analogous statement for a vector space $V$. Lemma~\ref{lem:VerAuto} essentially says that the exponential map
$$
\widehat{S}^{\geq 2}V\dual \otimes V \to \Aut_1(\widehat{S}V\dual), \quad Y \mapsto e^Y
$$
is surjective. Here, $\widehat{S}^{\geq 2}V\dual \otimes V$ represents the space of formal vector fields on $V$ with coefficients of degree greater than one, and $\Aut_1(\widehat{S}V\dual)$ consists of continuous automorphisms with a trivial linear component. Nevertheless, the surjectivity of the exponential map fails if one considers general formal vector fields and general automorphisms. Following are a couple of issues: 
\begin{enumerate}
\item
For $Y \in S^1 V\dual \otimes V$, the linear component of $e^Y$ is essentially the exponential map of the corresponding linear endomorphism of $V\dual$. If $\KK= \RR$, then since the eigenvalues of the exponential of a real matrix must be positive, the map $e^Y:\widehat{S}V\dual \to \widehat{S}V\dual$ cannot be an arbitrary automorphism.

\item
For $Y \in S^0V\dual \otimes V$, the map $e^Y:\widehat{S}V\dual \to \widehat{S}V\dual$ might not even be well-defined. For example, let $V = \RR^1$ and $Y = \frac{d}{dt}$, where $t$ is the standard coordinate on $\RR^1$.  Then the constant term of $e^Y(\sum_{n=0}^\infty t^n)$ is the divergent series $\sum_{n=0}^\infty 1$ which is not defined. 
\end{enumerate}
\end{remark}

\subsection{Proof of Theorem~\ref{thm:PhiIteration}}\label{sec:VerIsoFedosovDG}

To prove Theorem~\ref{thm:PhiIteration}, we need the following lemmas.

\begin{lemma}\label{lem:hpartialInjectivity}
Let $Q_1$ and $Q_2$ be two Fedosov homological vector fields on $L[1]\oplus B$, 
\begin{equation}\label{eq:DeltaQ}
\Delta Q = Q_1 - Q_2 = d_L^{\nabla_1} - d_L^{ \nabla_2} + X_1 - X_2 \in \Gamma(\Lambda^1 L\dual \otimes  \widehat{S}^{\geq 1} B\dual \otimes B),
\end{equation}
and $\eth:\sE \to \sE$ be the operator defined by 
\begin{equation}\label{eq:partial}
\eth(\phi) =  \Delta Q \circ \phi + [\koszul+Q_2 ,\phi], \qquad \forall\, \phi \in \sE.
\end{equation}
Then the operator $\id_{\sE} - \htpd \eth: \sE \to \sE$ is injective.
\end{lemma}
\begin{proof}
Let $\sE^{k,-q}$ be the space 
$$
\sE^{k,-q}:= \Gamma(\Lambda^\bullet L\dual \otimes S^k B\dual \otimes S^q B).
$$
Then, given any $N \in \ZZ$,
\begin{equation}\label{eq:Edecomposition}
\sE^{\geq N} = \prod_{q=0}^\infty \prod_{k=q+N}^\infty \sE^{k,-q} 
= \prod_{l=N}^\infty \sE^l, 
\end{equation}
where $\sE^l = \prod_{q=0}^\infty \sE^{q+l,-q} = \prod_{q=0}^\infty \Gamma(\Lambda^\bullet L\dual \otimes S^{q+l} B\dual \otimes S^q B)$. It is clear that the composition in $\End_\KK(\cR)$ respects the decomposition \eqref{eq:Edecomposition}:
$$
\circ: \sE^{l_1} \times \sE^{l_2} \to \sE^{l_1 + l_2}.
$$
Therefore, 
$$
\big(\htpd(\eth -[d_L^{\nabla_2},\argument])\big)(\sE^l) \subset \sE^{\geq l+1}.
$$
Furthermore, it is straightforward to show that 
$$
[d_L^{\nabla_2},\argument]: \Gamma(\Lambda^\bullet L\dual \otimes S^k B\dual \otimes S^q B) \to \Gamma(\Lambda^{\bullet +1} L\dual \otimes S^k B\dual \otimes S^q B),
$$
and hence
\begin{equation}\label{eq:hpartialDeg}
\htpd \eth(\sE^l) \subset \sE^{\geq l+1}.
\end{equation}

Now assume $D \in \ker(\id_{\sE}-\htpd \eth)$. The operator $D$ can be decomposed as $D = \sum_{l=N}^\infty D_l$, where $D_l \in \sE^l$. Thus,
\begin{equation}\label{eq:hpartialIteration}
D = \htpd \eth(D) = \sum_{l=N}^\infty \htpd \eth(D_l).
\end{equation}
By projecting \eqref{eq:hpartialIteration} to $\sE^N$, it follows from \eqref{eq:hpartialDeg} that $D_N =0$. The same argument inductively shows that each component $D_l$ vanishes. Thus $D=0$, and the injectivity of $\id_{\sE}-\htpd \eth$ follows.
\end{proof}

To justify the convergence of the series \eqref{eq:PhiSeries}, we need the following lemma:

\begin{lemma}\label{lem:OperatorSeriesConvergence}
Let $F_n:\sE \to \sE$, $n=0,1, \cdots$, be a sequence of linear operators. If there exists $N_0 \in \ZZ$ such that
$$
F_n(\sE^{\geq N}) \subset \sE^{\geq N_0 + N+n}, \qquad \forall\,  N\in \ZZ, \, n \in \NN_0,
$$
then the series 
$$
F = \sum_{n=0}^\infty F_n
$$
is a well-defined linear operator on $\sE$.
\end{lemma}
\begin{proof}
Given any element $\psi \in \sE$, since $\sE = \bigcup_{N \in \ZZ} \sE^{\geq N}$, we may assume $\psi \in \sE^{\geq N}$ for some $N\in \ZZ$. By the assumption, we have that $F_n(\psi) \in \sE^{\geq N_0 + N+n}$ for any $n \geq 0$. 

Given any $f=\sum_{k=0}^\infty f_k \in \Gamma(\Lambda^\bullet L\dual \otimes \widehat{S}B\dual)$ with $f_k \in \Gamma(\Lambda^\bullet L\dual \otimes S^k B\dual)$, write
$$
F_n(\psi)(f_k) = \sum_{l=N_0+N+n+k}^\infty f_{n,k,l} \in \cI^{N_0+N+n+k} = \Gamma(\Lambda^\bullet L\dual \otimes \widehat{S}^{\geq N_0+N+n+k} B\dual),
$$
where $f_{n,k,l}$ is the component in $ \Gamma(\Lambda^\bullet L\dual \otimes {S}^{l} B\dual)$. Therefore,
\begin{align*}
F(\psi)(f) & = \sum_{n,k=0}^\infty\sum_{l=N_0+N+n+k}^\infty f_{n,k,l} \\
& = \sum_{l=0}^\infty \left( \sum_{\substack{n+k = l-N_0 -N \\ n,k \geq 0 }} f_{n,k,l} \right).
\end{align*}
Since, for each $l$, there are only finite $n,k$ with the properties $n+k = l-N_0 -N$ and $n,k \geq 0$, we conclude that $F(\psi)(f)$ is a well-defined element in $\Gamma(\Lambda^\bullet L\dual \otimes \widehat{S}B\dual)$. The linearity of $F(\psi)$ is clear. The continuity of $F(\psi)$ follows from Lemma~\ref{lem:ContiCriterion} together with the observation that $F(\psi)(\cI^k) \subset \cI^{N_0+N+k}, \, \forall k \geq 0$.  Therefore, 
$$
F(\psi) \in \sE^{\geq N_0+N}.
$$
Since $\psi$ is arbitrary, the proof is complete. 
\end{proof}

Now we are ready to prove Theorem~\ref{thm:PhiIteration}.

\begin{proof}[Proof of Theorem~\ref{thm:PhiIteration}]
Let $Q_1$ and $Q_2$ be two Fedosov homological vector fields on $\cM=L[1]\oplus B$.
Suppose there exists a vertical isomorphism of Fedosov dg manifolds
$\phi:(\cR,Q_2)\to(\cR,Q_1)$, i.e.\ a continuous, $\Gamma(\Lambda^\bullet L^\vee)$-linear,
degree-preserving automorphism of the algebra $\cR$ such that
\begin{equation}\label{eq:IsoFedosovDG}
\phi\circ Q_2=Q_1\circ\phi
.\end{equation}
Equation~\eqref{eq:IsoFedosovDG} is equivalent to
\begin{equation}\label{eq:ver2IsoFedosovDG}
\Delta Q\circ\phi+[Q_2,\phi]=0
,\end{equation}
where $\Delta Q=Q_1-Q_2$.

Let $\phi = \sum_{q=0}^\infty D_q$ be the decomposition established in Proposition~\ref{prop:BddOp&DiffOp}. 
Note that, by Lemma~\ref{lem:VerAuto}, there exists a pair of an automorphism $\chi: B \to B$ of vector bundles and a vertical vector field $Y\in \Gamma(\widehat{S}^{\geq 2} B\dual \otimes B)$ such that $\phi=\chi^\natural\circ e^Y$. In particular, each $D_q$ is a $0$-form, i.e.\ $D_q \in \Gamma(\Lambda^0 L\dual \otimes \widehat{S} B\dual \otimes S^q B)$. Since the homotopy operator $\htpd$ vanishes on $0$-forms, we have that $\htpd(\phi) = \sum_{q=0}^\infty \htpd(D_q) =0$.

Since $\phi$ is a vertical isomorphism, it preserves the $\cI$-adic filtration by Remark~\ref{rmk:VerAlgIsoBddBelow}. This implies that $D_q \in \Gamma(\Lambda^0 L\dual \otimes \widehat{S}^{\geq q} B\dual \otimes S^q B)$, and thus $ \sigma_0(D_q) =0$ for $q \geq 1$. Furthermore, since $\phi(1)=1$, it follows from Equation~\eqref{eq:D_0inDecomposition} that $D_0 =1$. Consequently, we have that $\sigma_0(\phi) = \sum_{q=0}^\infty \sigma_0 (D_q) = 1$.

It follows from Corollary~\ref{cor:KoszulBddVO} and Equation~\eqref{eq:ver2IsoFedosovDG} that
\[ \phi=\sigma_0(\phi)+\delta \htpd(\phi)+\htpd\delta(\phi)
=1+\htpd\big([(\koszul+Q_2)-Q_2,\phi]\big)=1+\htpd\eth(\phi) ,\]
where $\eth$ is the operator defined by Equation~\eqref{eq:partial}. 
This computation proves that every vertical isomorphism $\phi:(\cR,Q_2)\to(\cR,Q_1)$
must satisfy Equation~\eqref{eq:PhiIteration}.
Since there is, according to Lemma~\ref{lem:hpartialInjectivity}, at most one $\cI$-adic filtration-shifting vertical operator
satisfying Equation~\eqref{eq:PhiIteration}, we have proved that, provided $\phi$ exists, it is unique.

Finally, since $(\htpd\eth)^k(\sE^{\geq N})\subset\sE^{\geq N+k}$ for all $k\in\NN_0$ and $N\in\ZZ$, it follows from Lemma~\ref{lem:OperatorSeriesConvergence} that the series \eqref{eq:PhiSeries} defines an operator $\phi$. It is clear that this operator $\phi$ satisfies Equation~\eqref{eq:PhiIteration}.
Therefore, the existence of the vertical isomorphism $\phi$ is established.
\end{proof}

\begin{remark}
According to Lemma~\ref{lem:VerAuto}, the isomorphism $\phi$ in Theorem~\ref{thm:PhiIteration} is of the form $\phi = \chi^\natural \circ e^Y$ for some vector bundle isomorphism $\chi:B \to B$ and some vertical vector field $Y \in \Gamma(\widehat{S}^{\geq 2} B\dual \otimes B)$. In fact, the isomorphism $\chi$ is the dual map of the first component $\phi_1:\Gamma(B\dual) \xto{\phi} \Gamma(\Lambda^\bullet L\dual \otimes \widehat{S}B\dual) \xto{\pr_1} \Gamma(S^1 B\dual)$ of $\phi$. Since $\pr_1 \circ (\htpd \eth)^k(1)\big|_{\Gamma(B\dual)} =0$ for any $k \geq 1$, it follows from \eqref{eq:PhiSeries} that the component $\phi_1$ is $\id$. Thus, we have $\chi = \id_B$ and $\phi = e^Y$.
\end{remark}

\begin{corollary}\label{cor:Y=LogPhi}
Every vertical isomorphism $\phi:(\cR,Q_2)\to(\cR,Q_1)$ of Fedosov dg manifolds
is necessarily the flow $\phi=e^Y$ of some vertical vector field
$Y\in\Gamma\big(\widehat{S}^{\geq 2}B^\vee \otimes B\big)$
and, indeed, we have
\begin{equation}\label{eq:Y=LogPhi}
Y=\log(\phi)=\sum_{k=1}^\infty\sum_{n_1,\cdots,n_k\geq 0}\frac{(-1)^{k+1}}{k}\cdot
(\htpd\eth)^{n_1}\big(\htpd(\Delta Q)\big)\circ
(\htpd\eth)^{n_2}\big(\htpd(\Delta Q)\big)\circ\dots\circ
(\htpd\eth)^{n_k}\big(\htpd(\Delta Q)\big),
\end{equation}
where $\Delta Q = Q_1 - Q_2$.
\end{corollary}

\subsection{Applications}

Given two different choices of a splitting of the short exact sequence \eqref{eq:splitting} and a torsion-free $L$-connection on $B$, one has two induced PBW isomorphisms $\pbw_1$ and $\pbw_2: \Gamma(SB) \to \frac{\cU(L)}{\cU(L)\Gamma(A)}$. Since the map 
\begin{equation}\label{eq:psiMap}
\psi:=\pbw_2\inv \circ \pbw_1: \Gamma(SB) \to \Gamma(SB)
\end{equation}
is an isomorphism of $R$-coalgebras, its dual map $\phi:\Gamma(\widehat{S}B\dual) \to \Gamma(\widehat{S}B\dual)$ is an isomorphism of $R$-algebras. Extending $\phi$ by $\Gamma(\Lambda^\bullet L\dual)$-linearity, we have a vertical automorphism $\id \otimes \phi$ of the Fedosov manifold $\cM$, which will be denoted by $\phi$ by abuse of notations. According to \cite[Section~2.2]{MR4325718}, the map 
\begin{equation}\label{eq:phiMap}
\phi= \id_{\Lambda^\bullet L\dual}\otimes (\pbw_2\inv\circ \pbw_1)\dual: \Gamma(\Lambda^\bullet L\dual \otimes \widehat{S} B\dual) \to \Gamma(\Lambda^\bullet L\dual \otimes \widehat{S} B\dual)
\end{equation}
is a vertical isomorphism of Fedosov dg manifolds from $(\cM,Q_1)$ to $(\cM,Q_2)$. 

Note that, since
\[ \pair{\phi(f)}{\theta}=\pair{f}{\psi(\theta)} \]
for all $\theta\in\Gamma(SB)$ and $f\in\Gamma\big(\widehat{S}(B^\vee)\big)$,
we have
\[ \pair{f}{\psi^{-1}(\theta)}=\pair{\phi^{-1}(f)}{\theta} \]
or, equivalently,
\begin{equation}\label{eq:something}
\psi^{-1}(\theta)=\theta\circ\phi^{-1}
.\end{equation}

\begin{corollary}\label{cor:PBWformula}
There exists a unique $Y \in \Gamma(\widehat{S}^{\geq 2} B\dual \otimes B)$, which is is given by the formula \eqref{eq:Y=LogPhi}, such that 
$$\pair{\phi(f)}{\theta} =\pair{e^Y (f)}{\theta}= \pair{f}{\pbw_2\inv \circ \pbw_1(\theta)}$$
for any $f \in \Gamma(\widehat{S}B\dual)$, $\theta \in \Gamma(SB)$. 
\end{corollary}

\subsubsection{Application to geodesic coordinate systems}

In the case $(L,A) = (T_M,M\times 0)$, the PBW map $\pbw^\nabla: \Gamma(ST_M) \to \cU(T_M)$ is the fiberwise $\infty$-order jet of the exponential map $\exp^\nabla:T_M \to M \times M$ along the zero section of $T_M$.  
In other words,  
$$
\pbw^\nabla(X_0 \odot \cdots \odot X_k)(f)\big|_p = X_0\big|_p X_1\big|_p \cdots X_k\big|_p( f \circ \exp_p^\nabla) 
$$
for all $X_0, \cdots , X_k \in \Gamma(T_M) = \Gamma(B)$, $f \in C^\infty(M)$ and $p\in M$. Here, the tangent vector $X_i\big|_p \in T_M\big|_p \cong T(T_M\big|_p)\big|_0$ is identified with the corresponding tangent vector of the vector space $T_M\big|_p$ at zero.

Assume we are given two connections $\nabla_1$ and $\nabla_2$ on $M$. Suppose $\exp_1$ and $\exp_2$ are the exponential maps associated with $\nabla_1$ and $\nabla_2$, respectively.  The map 
$$
\exp_{2,p}\inv \circ \exp_{1,p}: T_M \big|_p \to T_M \big|_p
$$  
is the coordinate transformation between two associated geodesic coordinate systems around a point $p \in M$. This transformation is related to $\pbw_2\inv \circ \pbw_1: \Gamma(ST_M) \to \Gamma(ST_M)$ by the following formula:
\begin{equation}
\pbw_2\inv \circ \pbw_1(X_1 \odot \cdots \odot X_k)(g)\big|_{0_p} = X_0\big|_p X_1\big|_p \cdots X_k\big|_p( g \circ \exp_{2,p}\inv \circ \exp_{1,p})
\end{equation}
for any $X_0, \cdots , X_k \in \Gamma(T_M) $ and $g \in C^\infty(T_M\big|_p)$.

\begin{corollary}
The $\infty$-order jet of the map $\exp_2\inv \circ \exp_1:T_M \to T_M$ along the zero section is given by $\pbw_2\inv \circ \pbw_1: \Gamma(ST_M) \to \Gamma(ST_M)$ which satisfies the equation 
$$
\pair{e^Y (f)}{\theta}= \pair{f}{\pbw_2\inv \circ \pbw_1(\theta)}, \qquad \forall \, f \in \Gamma(\widehat{S}T_M\dual), \, \theta \in \Gamma(ST_M).
$$
Here, $Y\in \Gamma(\widehat{S}^{\geq 2} T_M\dual \otimes T_M)$ is the vertical vector field given by the formula \eqref{eq:Y=LogPhi}.
\end{corollary}

\subsubsection{Application to the $L_\infty$ structure on polydifferential operators}

In \cite[Section~3.3]{MR4325718}, Bandiera, Sti\'{e}non and Xu proved the uniqueness of the $L_\infty$ structures on the polyvector fields and on the polydifferential operators of a given Lie pair $(L,A)$. The key lemma in their argument is \cite[Lemma~3.25]{MR4325718}: Let $\phi$ be the map defined in \eqref{eq:phiMap}. For $\eta^I \in \Gamma(S^{|I|}B\dual)$ and $\partial_{J_0} , \cdots , \partial_{J_k} \in \Gamma(SB)$, the map $\phi$ induces a pushforwarded polydifferential operator by the formula 
\begin{equation}\label{eq:PushedPolyDiffOp}
\phi_\ast(\eta^I \otimes \partial_{J_0} \otimes \cdots \otimes \partial_{J_k})(f_0, \cdots, f_k) = \phi\Big((\eta^I \otimes \partial_{J_0} \otimes \cdots \otimes \partial_{J_k})\big(\phi\inv(f_0), \cdots, \phi\inv(f_k) \big) \Big),
\end{equation}
for $f_0,\cdots, f_k \in C^\infty(\cM)$.

\begin{proposition}[{\cite[Lemma~3.25]{MR4325718}}] 
The pushforwarded polydifferential operator \eqref{eq:PushedPolyDiffOp} satisfies the equation 
\begin{equation}\label{eq:BSXLemma3.25}
\phi_\ast(\eta^I \otimes \partial_{J_0} \otimes \cdots \otimes \partial_{J_k}) = \eta^I \otimes \psi\inv(\partial_{J_0}) \otimes \cdots \otimes \psi\inv(\partial_{J_k}) + P,
\end{equation}
for some $P \in \Gamma(\widehat{S}^{>|I|}B\dual \otimes (S B)^{\otimes k+1})$. 
\end{proposition}

By Corollary~\ref{cor:PBWformula}, we can give an easy proof of Bandiera--Sti\'{e}non--Xu's equation \eqref{eq:BSXLemma3.25}: Note that 
\begin{align*}
\phi_\ast(\eta^I \otimes \partial_{J_0} \otimes \cdots \otimes \partial_{J_k}) & = \phi \circ \big(\eta^I \otimes (\partial_{J_0} \circ \phi\inv) \otimes \cdots \otimes (\partial_{J_k}\circ \phi \inv) \big) \\
& = (\id + e^Y - \id) \circ \big(\eta^I \otimes \psi\inv(\partial_{J_0}) \otimes \cdots \otimes \psi\inv(\partial_{J_k}) \big)  \qquad \text{by Equation~\eqref{eq:something}} \\
& = \eta^I \otimes \psi\inv(\partial_{J_0}) \otimes \cdots \otimes \psi\inv(\partial_{J_k}) + P,
\end{align*}
where 
$$
P = \sum_{k=1}^\infty \frac{1}{k!} Y^k \circ \big(\eta^I \otimes \psi\inv(\partial_{J_0}) \otimes \cdots \otimes \psi\inv(\partial_{J_k})\big).
$$
Since $Y \in \Gamma(\widehat{S}^{\geq 2} B\dual \otimes Y)$, we have $P \in \Gamma(\widehat{S}^{>|I|}B\dual \otimes (S B)^{\otimes k+1})$. This proves Equation~\eqref{eq:BSXLemma3.25}.

\subsubsection{Application to Kapranov dg manifolds}

In \cite{MR4271478}, Laurent-Gengoux, Sti\'{e}non and Xu introduced Kapranov dg manifolds $A[1] \oplus B$ associated with a Lie pair $(L,A)$. 
Here, a \textbf{Kapranov dg manifold} structure on $A[1] \oplus B$ is a homological vector field $Q$ on $A[1] \oplus B$ such that both the inclusion $(A[1],d_A) \into (A[1]\oplus B, Q)$ and the projection $(A[1] \oplus B, Q) \onto (A[1],d_A)$ are morphisms of dg manifolds. An \textbf{isomorphism of Kapranov dg manifolds} from $A[1] \oplus B$ to $A'[1] \oplus B'$ is a pair $(\Psi,\psi)$ of morphisms of dg manifolds $\Psi:A[1] \oplus B \to A'[1] \oplus B'$ and $\psi: A[1] \to A'[1]$ such that the following two diagrams commute:
$$
\begin{tikzcd}
A[1] \oplus B \ar[r,"\Psi"] & A'[1] \oplus B' \\
A[1] \ar[u,hook] \ar[r,"\psi"'] & A'[1]\ar[u,hook] 
\end{tikzcd}
\qquad 
\begin{tikzcd}
A[1] \oplus B \ar[r,"\Psi"] \ar[d,two heads]& A'[1] \oplus B' \ar[d,two heads] \\
A[1]  \ar[r,"\psi"'] & A'[1].
\end{tikzcd}
$$

According to \cite[Proposition~4.9]{MR4271478}, a Kapranov dg manifold structure on $A[1] \oplus B$ is equivalent to an $A$-action $\varrho:\Gamma(A) \times \Gamma(SB) \to \Gamma(SB)$ by coderivations such that $\varrho_a(1) = 0$ for all $a \in \Gamma(A)$. Therefore, the action 
$$
\varrho_a(s) = \pbw\inv(\ia(a)\cdot \pbw(s)), \qquad \forall \, a \in \Gamma(A), \, s \in \Gamma(SB),
$$
defines a Kapranov dg manifold structure on $A[1] \oplus B$, where $\ia:A \into L$ is the inclusion map.

Given two choices of a splitting of the short exact sequence \eqref{eq:splitting} and a torsion-free $L$-connection on $B$, we have two associated PBW maps $\pbw_1$ and $\pbw_2$ which induce two Kapranov dg manifold structures $Q_1$ and $Q_2$ on $A[1] \oplus B$. Since the automorphism 
$$
\pbw_2\inv \circ \pbw_1:\Gamma(SB) \to \Gamma(SB)
$$ 
of filtered $R$-coalgebras intertwines the coderivations $\pbw_2\inv\big(\ia(a)\cdot \pbw_2(\argument)\big)$ and  $\pbw_1\inv\big(\ia(a)\cdot \pbw_1(\argument)\big)$ for all $a \in \Gamma(A)$, it determines an isomorphism of Kapranov dg manifolds by \cite[Proposition~4.11]{MR4271478}. Also see \cite[Proposition~5.24~(2)]{MR4271478}.
Consequently, Corollary~\ref{cor:PBWformula} provides an explicit formula of an isomorphism of Kapranov dg manifolds between $(A[1]\oplus B, Q_1)$ and $(A[1]\oplus B, Q_2)$.

\bibliographystyle{plain}
\bibliography{refFedosov}

\end{document}